\documentclass{article}

\usepackage[english]{babel}
\usepackage{caption}
\usepackage{graphicx}
\usepackage[pdfborder={0 0 0}]{hyperref}
\usepackage[utf8]{inputenc}
\usepackage{amsmath, amssymb, amsthm}
\usepackage{hyperref}
\usepackage{subfiles}
\usepackage{enumitem}
\usepackage{tikz}
\usepackage{tikz-cd}
\usepackage{tocloft}
\usepackage{bbm}
\usepackage{stmaryrd}
\usepackage[nottoc]{tocbibind}
\usepackage{fullpage}
\usepackage{colonequals} 
     
\theoremstyle{plain}                            
\newtheorem{stelling}{Theorem}
\newtheorem{gevolg}[stelling]{Corollary}
\newtheorem{lemma}[stelling]{Lemma}
\newtheorem{prop}[stelling]{Proposition}

\theoremstyle{definition}
\newtheorem{definition}[stelling]{Definition}
\newtheorem{opm}[stelling]{Remark}

\title{CM-values of $p$-adic $\Theta$-functions}
\author{Michael A. Daas}

\date{\vspace{-5ex}}

\usepackage{microtype}

\begin{document}

\maketitle

\begin{center}
\textbf{Abstract}
\end{center}
\noindent
We prove a $p$-adic version of the work by Gross and Zagier \cite{GZ} on the differences between singular moduli by proving a set of conjectures by Giampietro and Darmon \cite{Giam}, who investigated the factorisation of a rational invariant associated to a pair of CM-points on a genus zero Shimura curve, obtained as the ratio of the CM-values of $p$-adic $\Theta$-functions. As did Gross and Zagier, we give two proofs; an algebraic proof using CM-theory, and more interestingly, also an analytic proof using $p$-adic infinitesimal deformations of Hilbert Eisenstein series in the style of \cite{DPV1, DPV2}. Since there are no explicit formulae for its cuspidal $p$-adic deformations, we instead compute the Frobenius traces of the appropriate Galois deformation, and show their modularity via an $R = T$ theorem. This approach aims to bridge the gap between classical CM-theory and the more recent $p$-adic advances in the theory of real multiplication.

\setcounter{tocdepth}{2}
\tableofcontents

\newpage

\section{Introduction}

In their paper \cite{GZ}, Gross and Zagier studied the differences between singular moduli, which are the CM-values of Klein's $j$-function. For example,
\begin{equation}
j \left( \frac{1 + \sqrt{-43}}{2} \right) - j \left( \frac{1 + \sqrt{-163}}{2} \right) = 2^{19} \cdot 3^6 \cdot 5^3 \cdot 7^3 \cdot 37 \cdot 433. \label{GZex}
\end{equation}
Aside from this number being rather smooth, one may observe that all its prime divisors are inert in both $\mathbb{Q}(\sqrt{-43})$ and $\mathbb{Q}( \sqrt{-163} )$. More precisely, all of these primes occur as the factor of a number of the form $43 \cdot 163 - x^2$ for some $|x| < \sqrt{43 \cdot 163}$, as the equality $43 \cdot 163 - 9^2 = 16 \cdot 433$ exemplifies. These patterns persist when repeating the experiment with other, possibly non-rational singular moduli if one takes the norm down to $\mathbb{Q}$. This paper studies factorisation phenomena that display a parallel with the observations made and subsequently fully explained by Gross and Zagier in \cite{GZ}.

The following notation will be used throughout. Fix two imaginary quadratic fields $K_1$ and $K_2$ with rings of integers $\mathcal{O}_1$ and $\mathcal{O}_2$ respectively. Let $D_1, D_2 < 0$ denote their discriminants and assume that they are coprime. We write $w_i = \# \mathcal{O}_i^{\times}$ for $i \in \{ 1, 2 \}$ and for any subset $S \subset F$, we let $S^+ \subset S$ denote the subset of totally positive elements of $S$. Write $D = D_1D_2$ and let $F = \mathbb{Q}( \sqrt{D} )$ be the real quadratic field and $L = \mathbb{Q}( \sqrt{D_1}, \sqrt{D_2})$ be the biquadratic field completing the following field diagram:
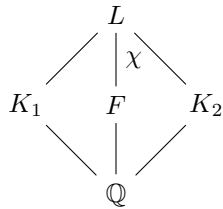
\begin{figure}[h]
\centering
 \begin{tikzpicture}[scale=0.6]

    \node (Q1) at (0,0) {$\mathbb{Q}$};
    \node (Q2) at (2,2) {$K_2$};
    \node (Q3) at (0,4) {$L$};
    \node (Q4) at (-2,2) {$K_1$};
    \node (Q5) at (0,2) {$F$};

    \draw (Q1)--(Q2); 
    \draw (Q1)--(Q4); 
    \draw (Q3)--(Q4);
    \draw (Q2)--(Q3);
    \draw (Q1)--(Q5);
    \draw (Q5)--(Q3) node [pos=0.5, right] {$\chi$};
    \end{tikzpicture}
\end{figure}

As the field extension $L/F$ is unramified at all finite places, it naturally induces a genus character $\chi \colon \text{Pic}(F)^+ \to \{ \pm 1 \}$. We let $\mathcal{D}_F$ denote the different ideal of $F$ and we define $$\rho(I) \colonequals \# \left\{ J \subset \mathcal{O}_L \mid \text{Nm}^L_F(J) = I \right\}$$ for any ideal $I \subset \mathcal{O}_F$. Finally, for any number field $M / \mathbb{Q}$ we let $G_M \colonequals \text{Gal}( \overline{\mathbb{Q}} / M )$ denote its absolute Galois group and $\mathcal{H}$ denotes the complex upper half plane. 

\subsection{The case of modular curves}

Suppose that $m$ is an integer supported at primes that are not inert in $F / \mathbb{Q}$. Assume that there is a unique prime $\ell$ both dividing $m$ an odd number of times, say $2k+1$ times, and with the additional property that any prime ideal $\mathfrak{l}$ of $F$ above $\ell$ satisfies $\chi( \mathfrak{l} ) = -1$. Further, let $\{ c_i \}$ be the set of exponents of primes dividing $m$ that split completely in $L$. Then we set
\[
F(m) = \ell^X \quad \text{where} \quad X = (k+1)\prod (c_i+1),
\]
and simply $F(m) = 1$ for all other $m \in \mathbb{Q}$. Finally, let $\tau_1, \tau_2 \in \mathcal{H}$ be CM-points of discriminants $D_1$ and $D_2$ respectively. Then Gross and Zagier proved in \cite{GZ} that
\begin{equation}
\text{Nm}_{\mathbb{Q}} \big( j( \tau_1 ) - j( \tau_2 ) \big)^{\frac{8}{w_1w_2}} = \pm \prod_{\substack{x^2 < D \\ x^2 \equiv D \text{ mod } 4}} F\left( \frac{D-x^2}{4} \right). \label{GZeq}
\end{equation}
Gross and Zagier gave two proofs of this formula, and the dissimilarities between these proofs cannot be overstated; whereas one made use of CM-theory, the other considered the diagonal restriction of a family $E_{1,\chi}(s)$, indexed by a complex parameter $s$, specialising to the non-holomorphic parallel weight $(1,1)$ Hilbert Eisenstein series attached to the character $\chi$, explicitly defined by
\begin{equation}
E_{1, \chi}(z_1,z_2) = \sum_{ \nu \in \mathcal{D}_F^{-1, +} } \bigg( \sum_{ I \mid (\nu) \mathcal{D}_F } \chi(I) \bigg) q^{\sigma_1(\nu)z_1 + \sigma_2(\nu)z_2} = \sum_{ \nu \in \mathcal{D}_F^{-1, +} } \rho( \nu \mathcal{D}_F ) q^{\sigma_1(\nu)z_1 + \sigma_2(\nu)z_2}, \label{HilEisdef}
\end{equation}
where $\sigma_1, \sigma_2 \colon F \to \mathbb{R}$ denote the two real embeddings of $F$. Even though this function vanishes identically, a fact which has historically been referred to as \emph{Hecke's sign error}, the family $E_{1,\chi}(s)$ is highly non-trivial and proved to be of great arithmetic importance. More precisely, they studied its first derivative with respect to the weight-parameter $s$, which must be a real analytic modular form of weight two for $\text{SL}_2(\mathbb{Z})$. One then applies the holomorphic projection operator $e^{\text{hol}}$ to conclude that
\begin{equation}
e^{\text{hol}}\left( \frac{d}{ds} \Delta E_{1,\chi}(s)\Big|_{s = 0} \right) \in M_2( \Gamma_0(1) ) = \{ 0 \}, \label{1linepf}
\end{equation}
where $\Delta$ denotes the diagonal restriction operator. On the other hand, it is possible to explicitly compute the Fourier coefficients of the expression on the left hand side, which with some careful analysis split up in two terms; one equal to the logarithm of the norm of $j(\tau_1) - j(\tau_2)$, and the other to the explicit formula that was to be proved. Notably, this proof does not use any CM-theory whatsoever.

The work \cite{GZ} of Gross and Zagier sparked further investigations that can be found in \cite{GZ2} and \cite{GKZ}. The former gave a relation between the heights of Heegner divisor classes on the Jacobian of modular curves and the first derivatives at $s = 1$ of the L-series of certain modular forms. The latter computed the height pairings of two distinct Heegner divisor classes to show that related quantities can be suitably combined to form the Fourier coefficients of a Jacobi form. These results expressed a strong analogy with the work of \cite{HZ}, which computed the intersection numbers of certain modular curves on Hilbert modular surfaces and related these to the coefficients of a weight 2 modular form. 

Later, by varying one of the discriminants instead, the heights of Heegner cycles were also shown in \cite{KRY1} to be connected to the derivative of a weight 3/2 Eisenstein series for $\text{SL}_2( \mathbb{Z} )$. The Kudla program aims to study the arithmetic properties of the first derivative certain Eisenstein series and to connect these with a specific class of arithmetic cycles and the special values of certain L-functions. Another relevant instance of a result in this direction can be found in \cite{schofer}.

In view of these results, the work \cite{GZ} can be regarded as the $X_0(1)$-case of the work done in \cite{GZ2} and \cite{GKZ}, the height pairing on whose Jacobian vanishes by virtue of the curve being of genus zero. Similarly, our main theorems will reflect the results in \cite{GZ} and we explain in Remark \ref{main} below how this result is to be interpreted in a more general framework as is done above.

\subsection{The case of Shimura curves}

Ever since the results of Gross and Zagier in \cite{GZ}, people have searched for generalisations of these kinds of factorisation phenomena. One place for such investigations has been the arithmetic of Shimura curves. Choose some $N \in \mathbb{N}$ and write $B_N$ for the quaternion algebra over $\mathbb{Q}$ with discriminant $N$. Assuming that $B_N$ is indefinite, it has a maximal order $R_N$ that is unique up to conjugation. Choosing a splitting $B_N \to M_2( \mathbb{R} )$, the subgroup $R_{N,1}^{\times} \subset R_N^{\times}$ consisting of all elements of unit norm can be regarded as acting on the complex upper half plane $\mathcal{H}$. The quotient $X_N = R_{N,1}^{\times} \setminus \mathcal{H}$ is compact and called a Shimura curve of level $N$. It has a model defined over $\mathbb{Q}$ and the Atkin-Lehner group $W_N$, generated by commuting involutions $w_r$ for every rational prime $r \mid N$, acts on it naturally. If $N \in \{ 6, 10, 22 \}$, the curve $X_N$ is of genus 0, and as such, its function field is generated by some function $j_N$. In contrast to the modular curve case, there is no cusp that we may use to normalise $j_N$ in a natural way. As such, there is no canonical choice for this function.

If all primes dividing $N$ are inert in $K_i$ for some $i \in \{ 1, 2 \}$, we can find (optimal) embeddings $\mathcal{O}_i \to R_N$ and for each such embedding, there is a unique point $P_i$ in $\mathcal{H}$ fixed by the image of the embeddings under our splitting $B_N \to M_2( \mathbb{R} )$, called the CM-point associated with the embedding. By Shimura's reciprocity law, as explained on the first pages of \cite{shim67}, the value $j_N(P_i)$ for a point $P_i \in X_N$ with complex multiplication by $\mathcal{O}_i$ is defined over the Hilbert class field $H_i$ of the field $K_i$. 

Elkies in \cite{elkies} numerically computed the CM-values for certain choices of a generator of the function field of certain Atkin-Lehner quotients of $X_N$, but not all values could be proved. However, the apparent smoothness of the resulting numbers did not go unnoticed. Using the theory of Borcherds lifts, Errthum in \cite{errthum} was able to prove the correctness of many of Elkies's computations, but no general conjectures as to the general structure of the values were posed. Some further explicit computations for particular choices of the generator of the function field can be found in \cite{voightshim} and more general rational points on Atkin-Lehner quotients are studied in \cite{clarkthesis}.

Instead of choosing a function $j_N$, one may observe that the cross-ratio of its values is well-defined and independent of any choices. We recall that for any distinct $x,y,z,w$ in some field, the cross-ratio is defined as
\[
[x,y,z,w] \colonequals \frac{z-x}{z-y} \cdot \frac{w-y}{w-x}.
\]
In 2022, Giampietro and Darmon in \cite{Giam} conducted extensive numerical computations with the quantities
\[
\frac{j_N(P_1) - j_N(P_2)}{j_N(P_1') - j_N(P_2)} \cdot \frac{j_N(P_1') - j_N(P_2')}{j_N(P_1) - j_N(P_2')},
\]
where for a CM-point $P$ on the curve $X_N$, we write $P' \colonequals w_p(\text{Frob}_p( P ))$ where $\text{Frob}_p$ denotes Frobenius at $p$ in the CM-field of definition for $P$. For example, Section 5 in \cite{Giam} elaborates on the example of $N = 6$, $D_1 = -43$ and $D_2 = -163$, in which it is computed that
\[
\text{Nm}^F_{\mathbb{Q}} \left[ \frac{j_N(P_1) - j_N(P_2)}{j_N(P_1') - j_N(P_2)} \cdot \frac{j_N(P_1') - j_N(P_2')}{j_N(P_1) - j_N(P_2')} \right] = \left( \frac{2 \cdot 29 \cdot 257 \cdot 277}{73 \cdot 137 \cdot 241} \right)^2.
\]
In parallel with Equation \ref{GZex}, one can check that all primes that occur on the right hand side are inert in both $K_1$ and $K_2$. More strongly, they are even prime divisors of a number of the form $43 \cdot 163 - x^2$ for some $|x| < \sqrt{43 \cdot 163}$, as the equality $43 \cdot 163 - 19^2 = 24 \cdot 277$ exemplifies. In fact, in this case, the authors did conjecture a general formula for this quantity. If we let $\{ a, -a, b, -b \}$ denote the four square roots of $D = D_1D_2$ modulo $2N$, and define
\[
\delta(x) = \begin{cases} +1 &\text{if } x \equiv \pm a \mod 2N; \\ -1 &\text{if } x \equiv \pm b \mod 2N,   \end{cases}
\]
then the following was conjectured in \cite{Giam}.
\begin{stelling}\label{giamconj1}
For any pair of embeddings $\mathcal{O}_i \to R_N$ for $i \in \{ 1,2 \}$, it holds that
\[
\emph{Nm}^{H_1H_2}_{\mathbb{Q}} \left[ \frac{j_N(P_1) - j_N(P_2)}{j_N(P_1') - j_N(P_2)} \cdot \frac{j_N(P_1') - j_N(P_2')}{j_N(P_1) - j_N(P_2')} \right]^{\frac{ \pm 2}{w_1w_2}} = \pm \prod_{\substack{x^2 < D \\ x^2 \equiv D \emph{ mod } 4N}} F\left( \frac{D-x^2}{4N} \right)^{\delta(x)}.
\]
\end{stelling}
The similarity with Equation \ref{GZeq} is apparent, even though, as our explicit examples show, the changed argument of the $F$-function causes most of the primes occurring in the factorisations to be very different in both cases. In the concluding section of \cite{Giam}, the computations from \cite{errthum} are shown to all be in accordance with the above result.

Using the $p$-adic uniformisation of Shimura curves, the authors of \cite{Giam} related this quantity to one of a $p$-adic nature as follows. From now on, we will write $N = pq$ for certain rational primes $p$ and $q$ and we will assume that both $p$ and $q$ are inert in both $K_1$ and $K_2$. This has the consequence that both $p$ and $q$ must split in $F / \mathbb{Q}$; we will denote these prime ideals as $\mathfrak{p}_1, \mathfrak{p}_2$ and $\mathfrak{q}_1, \mathfrak{q}_2$ respectively. Finally, we remark that $\mathfrak{p}_i$ and $\mathfrak{q}_i$ must be inert in $L / F$ for $i \in \{ 1, 2 \}$, and as such, it holds that $\chi( \mathfrak{p}_i ) = \chi( \mathfrak{q}_i ) = -1$.

Let $\mathcal{H}_p = \mathbb{P}^1(\mathbb{C}_p) \setminus \mathbb{P}^1(\mathbb{Q}_p)$ denote the $p$-adic upper half plane and let $B_q$ be the definite quaternion algebra over $\mathbb{Q}$ with discriminant $q$. By choosing a splitting $B_q \to M_2( \mathbb{Q}_p )$, we obtain an action of $B_q$ on $\mathcal{H}_p$. We let $R_q[1/p]$ be a maximal $\mathbb{Z}[1/p]$-order in $B_q$ and by $R_q[1/p]_1^{\times}$ we will denote its units of unit norm. The quotient $R_q[1/p]_1^{\times} \setminus \mathcal{H}_p$ is again compact. By the celebrated theorem of Cerednik and Drinfeld, originally proved in \cite{cerednik, drinfeld} and well explained in \cite{BouCar}, over $\mathbb{C}_p$ it is isomorphic to $X_N$, with the isomorphism itself being defined over $\mathbb{Q}_{p^2}$, the unique quadratic unramified extension of $\mathbb{Q}_p$.

The function fields of such curves are generated by so-called $\Theta$-functions, see \cite{GvdP}. Explicitly,
\[
\Theta(w_1,w_2; z) \colonequals \prod_{\gamma \in R_{q}[1/p]_1^{\times}} \frac{z - \gamma w_1}{z - \gamma w_2}.
\]
If $X_N$ is of genus 0, this describes a rational function on the quotient $R_q[1/p]_1^{\times} \setminus \mathcal{H}_p$ with divisor $2(w_1) - 2(w_2)$, the factor of 2 coming from the trivially acting element $-1 \in R_q[1/p]^{\times}_1$. For $i \in \{ 1,2 \}$, there exist (optimal) embeddings $\mathcal{O}_i \to R_q$ and for its image inside $M_2( \mathbb{Q}_p )$, there now exist two conjugate common fixed CM-points in $\mathcal{H}_p$. As explained in \cite{Giam}, if $\tau_i$ maps to $P_i$ under the Cerednik-Drinfeld isomorphism, then $\tau_i'$ will map to $P_i'$. Comparing divisors, we obtain the equality
\[
\prod_{\gamma \in R_{q}[1/p]_1^{\times}} [\gamma \tau_1, \gamma \tau_1', \tau_2, \tau_2'] = \frac{\Theta(\tau_1, \tau_1'; \tau_2)}{\Theta(\tau_1, \tau_1'; \tau_2')} = [ j_N( P_1 ), j_N( P_1' ), j_N( P_2 ), j_N( P_2' ) ]^2.
\]
The class group $\text{Pic}(K_i)$ acts naturally on the set of embeddings $\mathcal{O}_i \to R_q$. Let $\pi \in R_q$ be any quaternion with $\text{Nm}(\pi) = p$; up to conjugation there exist precisely $p+1$ such elements by Lemma 12 in \cite{Giam}.  If we now define
\[
\Theta( D_1, D_2 ) \colonequals \prod_{ \substack{ \text{Pic}(K_1) \cdot \tau_1 \\ \text{Pic}(K_2) \cdot \tau_2 } } \frac{\Theta(\tau_1, \tau_1'; \tau_2)}{\Theta(\tau_1, \tau_1'; \tau_2')} \quad \text{and} \quad \Theta_p( D_1, D_2 ) \colonequals \prod_{ \substack{ \text{Pic}(K_1) \cdot \tau_1 \\ \text{Pic}(K_2) \cdot \tau_2 } } \frac{\Theta(\tau_1, \tau_1'; \pi \tau_2)}{\Theta(\tau_1, \tau_1'; \pi \tau_2')},
\]
we also claim the following $p$-adic version of Theorem \ref{giamconj1}.
\begin{stelling}\label{giamconj2}
It holds that
\[
\left( \frac{\Theta( D_1, D_2 )}{\Theta_p( D_1, D_2 )} \right)^{\frac{\pm 2}{w_1w_2}} = \pm \prod_{\substack{x^2 < D \\ x^2 \equiv D \emph{ mod } 4N}} F\left( \frac{D-x^2}{4N} \right)^{\delta(x)}.
\]
\end{stelling}

\subsection{Parallels with RM-theory}

In the spirit of the original paper by Gross and Zagier \cite{GZ}, our approach to proving Theorems \ref{giamconj1} and \ref{giamconj2} is two-fold. First we present a direct proof of Theorem \ref{giamconj1} using CM-theory, which one could say is the standard approach for problems of this nature. It is not surprising that such a proof exists, and in fact, using the results from Phillips's thesis \cite{anphil}, the proof is rather straightforward.

Much more interesting is our second proof, which proves Theorem \ref{giamconj2} directly, not relying on any CM-theory whatsoever and is done purely by studying the infinitesimal $p$-adic deformation theory of the Galois representation associated with the $p$-stabilised parallel weight $(1,1)$ Hilbert Eisenstein series
\[
E^{(p)}_{1,\chi} \colonequals (1 - V_{\mathfrak{p}_1})( 1 + V_{\mathfrak{p}_2} ) E_{1, \chi}.
\]
Even though there are four choices for this $p$-stabilisation, we must choose one with opposite signs in order to ensure that $E^{(p)}_{1,\chi}$ will be a $p$-adic \emph{cuspform}. Theorem \ref{giamconj2} will be a consequence of the claim that the ordinary projection
\[
e^{\text{ord}}\left( \frac{d}{d\epsilon} \Delta E^{(p)}_{1,\chi}(\epsilon) \right) \in S_2( \Gamma_0(N) ),
\]
must vanish, once more strengthening the parallels with the analytic proof in \cite{GZ} and Equation \ref{1linepf}. 

Our main motivation for this second proof, which constitutes the focus of this paper, originates from the recent advancements in the theory of real multiplication. In \cite{pDV1}, Darmon and Vonk proposed a $p$-adic analogue of the differences between singular moduli as studied in \cite{GZ}; a certain \emph{rigid meromorphic cocycle} for the group $\text{SL}_2( \mathbb{Z}[1/p] )$, whose RM-values conjecturally display for real quadratic fields factorisations of similar intricacy to those considered hitherto. Certain cases of these conjectures are proved in the forthcoming work \cite{DV5}. This emerging theory should be well connected with many other areas of mathematics, notable among these the theory of Borcherds products and their ostensible connections to both $p$-adic heights and intersection numbers of geodesics on Shimura curves. More recently, these constructions were generalised to different quaternion algebras than the matrix algebra in \cite{gehrmann, XMX}, reflecting the step from modular curves to Shimura curves as above, and to more general orthogonal groups in \cite{gehrdarmon}.

Historically, the study of CM-theory has largely been facilitated by its connection to the geometry of abelian varieties and the moduli spaces that govern them. The development of an analogous RM-theory is complicated by the lack of such obvious connections to geometry. It is for this reason that the analytic proof in \cite{GZ} is of particular interest, as its independence from CM-theory contrasted strongly with the other, more algebraic, proof. Darmon, Pozzi and Vonk used similar ideas in \cite{DPV1, DPV2}, studying the ordinary projection of the diagonal restriction of the first derivative with respect to the weight of a $p$-adic family of Hilbert modular Eisenstein series attached to a more general odd character of the narrow class group of a real quadratic field, explicitly computing the Fourier coefficients of its ordinary projection. These quantities proved to be related to both Stark-Heegner points and Gross-Stark units, enriching the analogy between the classical theory of complex multiplication and its extension to real quadratic fields. Recently, Dasgupta and Kakde in \cite{DK1, DK2} proved Brumer-Stark conjecture away from 2 and used these ideas to prove the $p$-part of the integral Gross-Stark conjecture for the Brumer-Stark $p$-units in CM abelian extensions of a totally real field using the theory of group ring valued Hilbert modular forms.

The present work serves as a direct $p$-adic transposition of the analytic proof by Gross and Zagier in \cite{GZ} because we consider (an appropriate $p$-stabilisation of) the exact same Hilbert Eisenstein series $E_{1,\chi}$, using techniques that have recently also been deployed in the study of RM-theory. Secondly, the fact that we can give two proofs, one relying on the geometric moduli interpretation of the Shimura curve $X_N$ and one not relying on geometry at all, is interesting in view of the (presently still) unknown geometric framework within which the modern developments in RM-theory should best be described. 

Thirdly, our work constitutes an occurrence of a non-archimedean instance of the Kudla program, which is presently being investigated more intensively than ever. Even though it has classically been mostly studied in an archimedean context, recent years have seen some instances of similar results in non-archimedean settings. Examples of this include the results from \cite{DPV1, DPV2}, but also for instance the works \cite{darmontor} and  \cite{MHNic}. This emerging ``$p$-adic Kudla program'' still leaves much to be explored in the forthcoming years. It is also for this reason that in the present work, we do not explore the possibly third approach using Borcherds lifts in a similar style of \cite{errthum} when proving the CM-values from \cite{elkies}, even though the success of such an approach should be expected as well.

\subsection{Outline of the paper}

In Section \ref{proof1}, we describe an approach that mirrors the ideas behind Gross and Zagier's original algebraic proof in \cite{GZ}, exploiting the moduli interpretation of the Shimura curve $X_N$ and the theory of complex multiplication. We appeal to the main result of the PhD thesis of Andrew Phillips \cite{anphil}, which computes the degree of certain refinements of the moduli stack of certain false elliptic curves, following ideas of Howard and Yang in \cite{GZrefined}. Using these results, the proof of Theorem \ref{giamconj1} is rather straightforward.

The weight of our paper is concentrated in our second proof. For this, we follow the general strategy of the main arguments presented in \cite{DPV2} and \cite{DV5}, approaches that originated in \cite{DLR}. We study a $p$-stabilisation $E_{1, \chi}^{(p)}$ of the same Hilbert Eisenstein series $E_{1, \chi}$ as did Gross and Zagier in \cite{GZ}. The $p$-adic convergence of the infinite product defining $\Theta( D_1, D_2 )$ circumvents any regularisation arguments, facilitating swift computations. In this sense, our work is a true $p$-adic transposition of the work of Gross and Zagier in \cite{GZ}. Our second proof can be divided into three distinct steps, which we will now outline.

Since, unlike as in \cite{GZ}, the $q$-expansion of a $p$-adic family passing through $E_{1,\chi}^{(p)}$ is not a-priori known, we obtain such a family by deforming a rigidification $\rho_{\eta}$ of the decomposable representation $\mathbbm{1} \oplus \chi$. More precisely, we will consider all \emph{nearly ordinary} deformations, which are all such deformations for which the decomposition groups $G_{\mathfrak{p}_i} \subset G_F$ for $i \in \{ 1, 2 \}$ each fix a distinct line. This approach requires us to prove the modularity of such deformations to construct the required family. In this respect our argument is rather different from that of Gross and Zagier in \cite{GZ}, because they already had all the $q$-expansions they would need to carry out their arguments beforehand; in fact, this approach could not have possibly have worked in their archimedean setting. 

Therefore, Section \ref{modularity} proves an $R=T$ theorem, the first instance of which occurred in the proof by Wiles of Fermat's Last Theorem. Using similar methods as in Pozzi's thesis \cite{Pozzi} and the works \cite{BDP, BD, BDS, BelChen}, using fundamental results from Hida in \cite{hida1989, JAMI}, we construct a lift of $\rho_{\eta}$ to Hida's cuspidal nearly ordinary Hecke algebra, though some additional care is required to circumvent the difficulties of the cohomology groups $H^1( G_F, \mathbb{Q}_p( \chi ) )$ being 2-dimensional. Comparing the dimensions of the Hecke algebra and the resulting deformation ring, the $R=T$ theorem follows.

Finally, in Section \ref{proof2}, using a construction that associates to a quaternion an $\mathcal{O}_L$-ideal, we derive a bijection between the elements of $R_q[1/p]_1^{\times}$ and the set of $\nu \in ( \mathfrak{q}_1\mathcal{D}_F^{-1})^+$ of $p$-power trace counted with a multiplicity related to the function $\rho$ that also appears in Equation \ref{HilEisdef}. Then we consider one particular nearly ordinary deformation and explicitly compute the infinitesimal family of deformations of $E^{(p)}_{1, \chi}$ that corresponds to it. After taking its derivative with respect to the weight parameter and applying the ordinary projection operator, we argue why the result must vanish identically. Ultimately, we conclude the proof of Theorem \ref{giamconj2} by computing explicitly the coefficients of the (usually mostly theoretically used) ordinary projection and equating the first of these coefficients to zero.

\begin{opm}\label{main}
If we relax the condition that the Shimura curve $X_N$ be of genus zero, then the quotient $\Theta( D_1, D_2 ) / \Theta_p( D_1, D_2 )$ from Theorem \ref{giamconj2} can no longer be expected to be algebraic and indeed it generally will not be, for it will consist of both an algebraic part, determined above, and a transcendental part given by an appropriate $p$-adic height pairing on the Jacobian of the Shimura curve $X_N$, which vanishes in the genus zero case. Define for $i \in \{ 1, 2 \}$ the divisors on $X_N$ by the formulas
\[
\mathcal{D}_i = \sum_{ [c_i] \in \text{Pic}(K_i) } [c_i] \cdot ( P_i - P_i' ).
\]
Let $T_m$ denote the natural Hecke correspondence on the Jacobian of $X_N$ and let $(-,-)_p$ denote the $p$-adic height pairing as computed in \cite{grosslocal, werner}. Even though Werner's result in \cite{werner} only pertains to the quotient by Schottky groups, using the results from Section 4 of \cite{vdp92}, this may be extended to quotients by groups such as $R_q[1/p]^{\times}_1$. In the author's PhD thesis, an equality of the form
\[
e^{\text{ord}}\left( \frac{d}{d\epsilon} \Delta E^{(p)}_{1,\chi}(\epsilon) \right) = \sum_{m \geq 1} ( \mathcal{D}_1, T_m \mathcal{D}_2 )_p \ q^m \in S_2( \Gamma_0(N) ).
\]
will be proved. This $p$-adic instance of the Kudla program bears resemblance to various previous works in an archimedean setting; most notably to Theorem V.1 in \cite{GKZ}.
\end{opm}

\bigskip

\textbf{Acknowledgements:} I would like to thank Jan Vonk for bringing this problem to my attention, for his frequent conversations with me about the present work and also for his invaluable insights, without which this endeavour would never have been initiated. This work was supported in part by the VIDI Grant 213.084 and the ERC Starting Grant 101076941.

\newpage

\section{Algebraic proof}\label{proof1}

Relying on the results from the PhD thesis written by Andrew Phillips \cite{anphil}, we present our first proof of Theorem \ref{giamconj1}. The Shimura curve $X_N = R_{N,1}^{\times} \setminus \mathcal{H}$ is the coarse moduli space for isomorphism classes of \emph{false elliptic curves} $(A,\iota)$, where $A$ is a complex abelian surface and $\iota \colon R_N \to \text{End}(A)$ is an embedding of algebras. To formulate the main result from \cite{anphil}, we describe the fine moduli space now.

\subsection{Stacks and Arkelov degrees}

A \emph{false elliptic curve} over a scheme $S$ is a pair $(A, \iota)$ where $A \to S$ is an abelian scheme of relative dimension 2 and $\iota \colon R_N \to \text{End}_S(A)$ is a ring homomorphism. For $i \in \{ 1,2 \}$, a false elliptic curve over an $\mathcal{O}_L$-scheme $S$ with complex multiplication by $\mathcal{O}_i$ is a triple $(A, \iota, \kappa)$ where $(A,\iota)$ is a false elliptic curve over $S$ and $\kappa \colon \mathcal{O}_i \to \text{End}_{R_N}(A)$ is a ring map such that the action on the Lie algebra is through the natural structure map $\mathcal{O}_i \to \mathcal{O}_L \to \mathcal{O}_S(S)$. 

Let $\mathcal{M}$ be the algebraic stack, regular and flat of relative dimension 1 over $\text{Spec}(\mathcal{O}_L)$, such that $\mathcal{M}(S)$ for any $\mathcal{O}_L$-scheme $S$ denotes the category of false elliptic curves $(A,\iota)$ over $S$ satisfying a technical property regarding the Lie algebra, which can be found as Equation 1.2.1 in \cite{anphil}. This 2-dimensional stack $\mathcal{M}$ is usually referred to as (the integral model of) a Shimura curve. We are interested in two particular substacks of this stack; those defining the false elliptic curves with complex multiplication by $\mathcal{O}_i$ for $i \in \{ 1,2 \}$.

Let $\mathcal{Y}_i$ for $i \in \{ 1, 2 \}$ be the algebraic stack over $\text{Spec}(\mathcal{O}_L)$ with $\mathcal{Y}_i(S)$ the category of false elliptic curves over the $\mathcal{O}_L$-scheme $S$ with complex multiplication by $\mathcal{O}_i$. By forgetting the CM-structure, we have a morphism of stacks $\mathcal{Y}_i \to \mathcal{M}$. We further define $\mathcal{J} \colonequals \mathcal{Y}_1 \times_{\mathcal{M}} \mathcal{Y}_2$. By definition of the pullback of stacks, $\mathcal{J}$ now denotes the algebraic stack over $\text{Spec}(\mathcal{O}_L)$ with $\mathcal{J}(S)$ the category of triples $(\textbf{A}_1,\textbf{A}_2,f)$ where $\textbf{A}_i = (A_i, \iota_i, \kappa_i)$ for $i \in \{ 1,2 \}$ is a false elliptic curve over the $\mathcal{O}_L$-scheme $S$ with complex multiplication by $\mathcal{O}_i$ and where $f \colon \textbf{A}_1 \to \textbf{A}_2$ is an isomorphism. 

Following \cite{anphil}, we proceed to refine the stack $\mathcal{J}$ by associating to every object $(\textbf{A}_1,\textbf{A}_2,f) \in \mathcal{J}(S)$ a pair of objects $( \vartheta, \nu )$ as follows. It is well-known that for any positive integer $N$, there exists a unique ideal $\mathfrak{m}_N \subset R_N$ of index $N^2$. For $i \in \{ 1,2 \}$, there is a unique surjective ring map $\theta_i \colon \mathcal{O}_i \to R_N / \mathfrak{m}_N$ making the following diagram
\begin{figure}[ht]
\centering
\begin{tikzcd}
\mathcal{O}_i \arrow[rd, "\theta_i"] \arrow[rr] &              & \text{End}_{R_N/\mathfrak{m}_N}( A_i[\mathfrak{m}_N] ) \\
                             & R_N/\mathfrak{m}_N \arrow[ru] &  
\end{tikzcd}
\end{figure}
commute, where $A_i[\mathfrak{m}_N]$ denotes group scheme of the $\mathfrak{m}_N$-torsion inside $A_i$. 
 
Since $\mathcal{O}_L = \mathcal{O}_1 \otimes_{\mathbb{Z}} \mathcal{O}_2$, we obtain a well-defined surjective ring map $\vartheta \colon \theta_1 \otimes \theta_2 \colon \mathcal{O}_L \to R_N / \mathfrak{m}_N$. For brevity, we will denote $\mathcal{V} \colonequals \text{Hom}(\mathcal{O}_L, R/\mathfrak{m}_N )$. We let $\mathfrak{a}_{\vartheta} = \text{ker}( \vartheta ) \cap \mathcal{O}_F$ be the \emph{reflex ideal}. Since $\text{ker}(\vartheta)$ is an $\mathcal{O}_L$-ideal of norm $N^2$, it follows that $\mathfrak{a}_{\vartheta}$ is an $\mathcal{O}_F$-ideal of norm $N$. As such, if $N = pq$, there are precisely four possibilities for $\mathfrak{a}_{\vartheta}$;
\[
\mathfrak{a}_{\vartheta} \in \{ \mathfrak{p}_1\mathfrak{q}_1,  \mathfrak{p}_1\mathfrak{q}_2,  \mathfrak{p}_2\mathfrak{q}_1,  \mathfrak{p}_2\mathfrak{q}_2 \} \equalscolon \mathcal{I}.
\]
Next, as in Proposition 2.3 in \cite{GZrefined}, one can construct a map $\text{deg}_{\text{CM}} \colon \text{Hom}_{R_N}( A_1, A_2 ) \to \mathcal{D}_F^{-1}$ satisfying the defining property that $\text{Tr}^{F}_{\mathbb{Q}} ( \text{deg}_{\text{CM}}(f) ) = \text{deg}^*(f)$, where $\text{deg}^*(f)$ denotes the \emph{false degree} of the morphism $f$ as in Definition 2.2.15 in \cite{anphil}, which satisfies the property that $\text{deg}^*(f) = 1$ for all isomorphisms $f$. As such, we may consider the element $\nu = \text{deg}_{\text{CM}}( f ) \in \mathcal{D}_F^{-1}$. 

For any $\vartheta \in \mathcal{V}$, we define the stack $\mathcal{X}_{\vartheta}$ to be the algebraic stack over $\text{Spec}( \mathcal{O}_L )$ with $\mathcal{X}_{\vartheta}(S)$ for any $\mathcal{O}_L$-scheme $S$ the category of triples $(\textbf{A}_1, \textbf{A}_2, f) \in \mathcal{J}(S)$ with the property that the pair $(\textbf{A}_1, \textbf{A}_2)$ induces the map $\vartheta \in \mathcal{V}$ by the construction outlined above. For any $\nu \in \mathcal{D}_F^{-1}$, we let $\mathcal{X}_{\vartheta, \nu}$ denote the algebraic stack over $\text{Spec}( \mathcal{O}_L )$ with $\mathcal{X}_{\vartheta, \nu}(S)$ for any $\mathcal{O}_L$-scheme $S$ the category of triples $(\textbf{A}_1, \textbf{A}_2, f) \in \mathcal{X}_{\vartheta}(S)$ with the property that $\text{deg}_{\text{CM}}(f) = \nu$ on every component of $S$. 

We then obtain the decompositions
\[
\mathcal{J} = \bigsqcup_{ \vartheta \in \mathcal{V} } \mathcal{X}_{\vartheta} \quad \text{and} \quad \mathcal{X}_{\vartheta} = \bigsqcup_{\substack{ \nu \in \mathcal{D}_F^{-1} \\ \text{Tr}(\nu)=1 } } \mathcal{X}_{\vartheta,\nu}.
\]
The main result of \cite{anphil} concerns the \emph{Arkelov degree} of the stacks $\mathcal{X}_{\vartheta}$, which is defined as
\[
\text{deg}(\mathcal{X}_{\vartheta}) \colonequals \sum_{\mathfrak{r} \subset \mathcal{O}_L} \text{log}(|\mathbb{F}_{\mathfrak{r}}|) \sum_{x \in \mathcal{X}_{\vartheta}(k)} \frac{\text{length}(\mathcal{O}^{\text{sh}}_{\mathcal{X}_{\vartheta,x}})}{|\text{Aut}(x)|},
\]
where $k = \overline{\mathbb{F}}_{\mathfrak{r}}$ and where $\mathcal{O}^{\text{sh}}_{\mathcal{X}_{\vartheta,x}}$ denotes the strictly Henselian local ring of $\mathcal{X}_{\vartheta}$ for the \'etale topology at the geometric point $x$. By the decomposition above, we have
\begin{equation}
\text{deg}(\mathcal{X}_{\vartheta}) = \sum_{\substack{ \nu \in \mathcal{D}_F^{-1} \\ \text{Tr}(\nu)=1 } } \text{deg}( \mathcal{X}_{\vartheta,\nu} ). \label{stacksplit}
\end{equation}
Lastly, we define the finite set
\[
\text{Diff}_{\vartheta}(\nu) = \text{Diff}_{ \mathfrak{a}_{\vartheta} }( \nu ) \colonequals \{ \mathfrak{r} \subset \mathcal{O}_F \ | \ \chi_{\mathfrak{r}}(\nu \mathfrak{a}_{\vartheta}^{-1} \mathcal{D}_F) = -1 \},
\]
where $\chi_{\mathfrak{r}}$ denotes the character defined by the unramified extension of local fields $L_{\mathfrak{r}} / F_{\mathfrak{r}}$. Theorem 2 in \cite{anphil} then says the following.

\begin{stelling}\label{philthm}
Suppose that $\emph{Diff}_{\vartheta}(\nu) = \{ \mathfrak{r} \}$ for some prime $\mathfrak{r} \subset \mathcal{O}_F$. If $r \nmid N$, the degree of $\mathcal{X}_{\vartheta,\nu}$ satisfies
\[
\emph{exp}(\emph{deg}(\mathcal{X}_{\vartheta,\nu})) = r^{t_r/2} \quad \text{where} \quad t_r = \emph{ord}_{\mathfrak{r}}(\nu \mathfrak{r} \mathcal{D}_F) \cdot \rho(\nu \mathfrak{a}_{\vartheta} \mathfrak{r}^{-1}\mathcal{D}_F).
\]
If $r \mid N$, depending on whether $\mathfrak{r}$ divides $\mathfrak{a}_{\vartheta}$ or not, we must replace the term $\emph{ord}_{\mathfrak{r}}(\nu \mathfrak{r} \mathcal{D}_F)$ by $\emph{ord}_{\mathfrak{r}}(\nu)$ or $\emph{ord}_{\mathfrak{r}}(\nu \mathfrak{r})$ respectively. If $\nu \notin \mathcal{D}_F^{-1}$ or $\# \emph{Diff}_{\vartheta}(\nu) \neq 1$, then the degree is always 0.
\end{stelling}
This result gives us an explicit formula for the Arkelov degrees of the stacks $\mathcal{X}_{\vartheta,\nu}$ and as such, also of the degrees of the stacks $\mathcal{X}_{\vartheta}$. It is also clear from the result that the degree of the stack $\mathcal{X}_{\vartheta, \nu}$ only depends on the ideal $\mathfrak{a}_{\vartheta} \in \mathcal{I}$ and not on the precise map $\vartheta \in \mathcal{V}$. This allows us to define for any $\mathfrak{a} \in \mathcal{I}$ the quantity
\[
X( \mathfrak{a}, \nu ) \colonequals \text{deg}(\mathcal{X}_{\vartheta,\nu})
\]
where $\vartheta \in \mathcal{V}$ is arbitrary such that $\mathfrak{a}_{\vartheta} = \mathfrak{a}$. In the next subsection we will show that these expressions, when combined appropriately, constitute the right hand side of Theorem \ref{giamconj1}. The remainder of this section aims to relate the degrees of the stacks $\mathcal{X}_{\vartheta}$ to the left hand side, ultimately establishing equality.

\subsection{An elementary formula}

We assign a sign to each homomorphism $\vartheta \in \mathcal{V}$. This is done by recording the $\text{Gal}(F / \mathbb{Q})$-orbit of its associated ideal $\mathfrak{a}_{\vartheta} \in \mathcal{I}$. Explicitly, we set
\[
\delta(\vartheta) = \delta( \mathfrak{a}_{\vartheta} ) \colonequals \begin{cases} +1 & \text{ if } \mathfrak{a}_{\vartheta} \in \{ \mathfrak{p}_1\mathfrak{q}_1,  \mathfrak{p}_2\mathfrak{q}_2 \}; \\ -1 & \text{ if } \mathfrak{a}_{\vartheta} \in \{ \mathfrak{p}_1\mathfrak{q}_2,  \mathfrak{p}_2\mathfrak{q}_1 \}. \end{cases}
\] 
The following proposition connects the results from Theorem \ref{philthm} to the formula from Theorem \ref{giamconj1}.
\begin{prop}\label{thm1}
Let $\nu \in \mathcal{D}_F^{-1,+}$ with $\emph{Tr}(\nu) = 1$.  Then we can write $\nu = (x+\sqrt{D})/2\sqrt{D}$ for some integer $x$ with $x^2 < D$. Furthermore,
\[
F\left( \frac{D - x^2}{4N} \right)^{\delta(x)} = \prod_{ \mathfrak{a} \in \mathcal{I} } \emph{exp}\big(\delta(\mathfrak{a}) X( \mathfrak{a}, \nu ) \big).
\]
For other $\nu \in F$, both sides of the equation simply equal 1.
\end{prop}
\begin{proof}
Albeit elementary, we include this proof in some detail because we will require similar considerations later in our analytic approach as well. Suppose first that $N \nmid \text{Nm}( \nu \mathcal{D}_F )$, so no $\mathfrak{a} \in \mathcal{I}$ can divide $\nu \mathcal{D}_F$. Then by Theorem \ref{philthm}, we have $X( \mathfrak{a}, \nu ) = 0$. On the other hand, the $F$-value of the non-integer $(D-x^2)/(4N) = \text{Nm}( \nu \mathcal{D}_F ) / N$ is defined to be 1 as well. We may thus assume that $N \mid \text{Nm}( \nu \mathcal{D}_F )$ and as such, some $\mathfrak{a} \in \mathcal{I}$ divides $\nu\mathcal{D}_F$. We claim that this ideal is unique within $\mathcal{I}$. Indeed, if not, then either $p$ or $q$ would divide $\nu\mathcal{D}_F$. But this ideal is generated by $( x + \sqrt{D})/2$; a contradiction. Hence we may restrict our view to the unique $\mathfrak{a} \in \mathcal{I}$ dividing $\nu\mathcal{D}_F$ and we reduce to proving
\[
F\left( \frac{D - x^2}{4N} \right)^{\delta(x)} = \text{exp}\big(\delta(\mathfrak{a}) X( \mathfrak{a}, \nu ) \big).
\] 
Next we claim that, if chosen appropriately, the signs $\delta(x)$ and $\delta(\mathfrak{a})$ will agree for all $x$ and $\mathfrak{a}$. Indeed, if $\sigma$ denotes the non-trivial element of $\text{Gal}(F / \mathbb{Q})$, then we note that 
\[
(x+\sqrt{D})/2 \in \mathfrak{p}_1\mathfrak{q}_1 \iff (-x+\sqrt{D})/2) \in \sigma(\mathfrak{p}_1\mathfrak{q}_1) = \mathfrak{p}_2\mathfrak{q}_2.
\]
In other words, the $\text{Gal}(F / \mathbb{Q})$-orbit of the ideal $\mathfrak{a}$ dividing $(x+\sqrt{D})/2$ depends only on the $\{ \pm 1 \}$-orbit of the root $x$ represents of $D \mod 2N$; whence we may choose the signs such that they agree. We thus reduce to proving
\[
F\left( \frac{D - x^2}{4N} \right) = \text{exp}( X( \mathfrak{a}, \nu ) ).
\]
From the fact that the ideal $\nu \mathfrak{a}^{-1}\mathcal{D}_F$ is primitive, it easily follows that $\text{Diff}_{ \mathfrak{a} }(\nu)$ contains prime ideals $\mathfrak{r}$ above precisely those rational primes $r$ dividing $\text{Nm}( \nu \mathfrak{a}^{-1}\mathcal{D}_F ) = (D-x^2)/4N$ an odd number of times that are neither inert in $F$ nor completely split in $L$. This shows that, in case $\# \text{Diff}_{ \mathfrak{a} }(\nu) \neq 1$, by Theorem \ref{philthm} and the definition of $F$, both sides of the equation once again equal 1.

We thus suppose that $\text{Diff}_{ \mathfrak{a} }(\nu) = \{ \mathfrak{r} \}$ for some prime $\mathfrak{r} \subset \mathcal{O}_F$. If $r \nmid N$, we use Theorem \ref{philthm} to reduce to proving that
\[
F\left( \frac{D - x^2}{4N} \right) = r^{t_r/2} \quad \text{where} \quad t_r = \text{ord}_{\mathfrak{r}}(\nu \mathfrak{r} \mathcal{D}_F) \cdot \rho(\nu \mathfrak{a}^{-1} \mathfrak{r}^{-1}\mathcal{D}_F).
\]
To see this, we recall that when computing $F((D-x^2)/4N)$, the value $t_r$ is computed as the product of two factors, the first of which equals the number of times $r$ divides $(D-x^2)/4N$ plus 1; using that $r \nmid N$, this is precisely $\text{ord}_{\mathfrak{r}}(\nu \mathfrak{r} \mathcal{D}_F)$. For the second contribution, we first remark that $D-x^2$ is indeed supported on primes that are not inert in $F$, because by construction $D$ will be a square modulo each prime dividing this number. The remaining primes split into two categories; they are either inert in $L/F$, or split. By assumption, $\mathfrak{r}$ is the only of the former category dividing $\nu \mathfrak{a}_{\vartheta} \mathcal{D}_F$ an odd number of times. Now note that $\rho$ is a multiplicative quantity so that it can be computed prime by prime. If a prime $\mathfrak{s}$ of $F$ is inert in $L/F$, then $\mathfrak{s}^{2k}$ is uniquely a norm from $L$ for any positive integer $k$. If instead $\mathfrak{s}$ splits into $\mathfrak{S}_1$ and $\mathfrak{S}_2$ in $L$, then the ideal $\mathfrak{s}^k$ is a norm from $L$ in precisely $k+1$ ways; indeed, only the ideals $\mathfrak{S}_1^{k-i}\mathfrak{S}_2^i$ for $0 \leq i \leq k$ have the required norm. Combining all of this proves the claimed equality.

Finally, we must consider the case of $r \mid N$. We claim that $\mathfrak{r} \mid \mathfrak{a}$. Indeed, if not, then $\mathfrak{r} \mid \nu \mathcal{D}_F$, but because also $\mathfrak{a} \mid \nu \mathcal{D}_F$, it would follow that either $p$ or $q$ divides this primitive ideal; a contradiction. We thus use Theorem \ref{philthm} to reduce to proving that
\[
F\left( \frac{D - x^2}{4N} \right) = r^{t_r/2} \quad \text{where} \quad t_r = \text{ord}_{\mathfrak{r}}(\nu) \cdot \rho(\nu \mathfrak{a}^{-1} \mathfrak{r}^{-1}\mathcal{D}_F).
\]
The proof of the correctness of the factor $\rho$ can be copied verbatim from above. For the other factor, if $2k+1$ is the multiplicity with which $\mathfrak{r}$ divides $\nu \mathfrak{a} \mathcal{D}_F$, then it divides $\nu$ precisely $2k$ times as we assume $\text{gcd}(N , D) = 1$. On the other hand, we note that $2k$ is equal to the multiplicity with which $r$ divides the norm $(D-x^2)/4$ of $\nu \mathcal{D}_F$. Hence the number $(D-x^2)/(4N)$ contains precisely $2k-1$ factors of $r$; adding 1 back yields $2k$ again, completing the proof in this case too.
\end{proof}
\begin{gevolg}\label{thm1cor}
Suppose that $\vartheta, \vartheta' \in \mathcal{V}$ are such that $\delta( \vartheta ) \neq \delta( \vartheta' )$. Then
\[
\emph{exp}\big( \delta( \vartheta ) \emph{deg}( \mathcal{X}_{\vartheta} ) \big)^2 \cdot \emph{exp}\big( \delta( \vartheta' ) \emph{deg}( \mathcal{X}_{\vartheta'} ) \big)^2 = \prod_{\substack{ x^2 < D \\  x^2 \equiv D \mod 4N } } F\left( \frac{D - x^2}{4N} \right)^{\delta(x)}.
\]
\end{gevolg}
\begin{proof}
We apply Proposition \ref{thm1} and take the product over all totally positive $\nu \in F$ with $\text{Nm}( \nu \mathcal{D}_F )$ divisible by $N$; this yields the correct right hand side. For the left hand side, by Equation \ref{stacksplit}, we need merely observe that by Theorem \ref{philthm}, it holds that $X( \mathfrak{a}, \nu ) = X( \sigma( \mathfrak{a} ), \sigma( \nu ) )$ and as such, $\text{deg}( \mathcal{X}_{\vartheta} )$ does not depend on the $\text{Gal}(F/\mathbb{Q})$-orbit of its reflex ideal.  
\end{proof}

\subsection{Intersection theory}

It is clear from Corollary \ref{thm1cor} that to complete the proof, we must give an alternative description of the quantity $\text{deg}( \mathcal{X}_{\vartheta} )$. The most essential term in the formula defining its Arkelov degree is the $\text{length}( \mathcal{O}^{\text{sh}}_{\mathcal{X}_{\vartheta, x}} )$, which we aim to relate to an intersection on the coarse moduli space $X_N$. 

By definition, $\mathcal{J} = \mathcal{Y}_1 \times_{\mathcal{M}} \mathcal{Y}_2$; whence for any $x \in \mathcal{J}(k)$,
\[
\mathcal{O}^{\text{sh}}_{\mathcal{J}_{x}} \cong \mathcal{O}^{\text{sh}}_{\mathcal{\mathcal{Y}}_{1,x}} \otimes_{\mathcal{O}^{\text{sh}}_{\mathcal{\mathcal{M}}_{x}}} \mathcal{O}^{\text{sh}}_{\mathcal{\mathcal{Y}}_{2,x}}.
\]
By Theorem 4.1.3 in \cite{anphil}, which says that $|\text{Aut}(x)| = w_1w_2$ for all points $x \in \mathcal{J}(k)$, we may write
\[
\text{deg}(\mathcal{X}_{\vartheta}) = \frac{1}{w_1w_2} \sum_{\mathfrak{r} \subset \mathcal{O}_L} \log(|\mathbb{F}_{\mathfrak{r}}|) \sum_{x \in [\mathcal{X}_{\vartheta}(k)]} \text{length}(\mathcal{O}^{\text{sh}}_{\mathcal{\mathcal{Y}}_{1,x}} \otimes_{\mathcal{O}^{\text{sh}}_{\mathcal{\mathcal{M}}_{x}}} \mathcal{O}^{\text{sh}}_{\mathcal{\mathcal{Y}}_{2,x}}).
\]
Let $Y_i \to X_N$ be the closed subscheme supported on the points with complex multiplication by $\mathcal{O}_i$. As is outlined in Section II of \cite{vistoli}, we have a natural map $\pi \colon \mathcal{M} \to X_N$. As $X_N$ is smooth and $\mathcal{M}$ is a Deligne-Mumford stack, $\pi$ must be flat. Let $(Y_1 \times Y_2)_{\vartheta}$ denote the set of pairs of CM-points that induce $\vartheta \in \mathcal{V}$. The following lemma will move us away from the language of stacks and into the realm of schemes. 
\begin{lemma}\label{fstacks}
Fix a prime ideal $\mathfrak{r} \subset \mathcal{O}_L$ and a geometric point $x = ( \textbf{A}_1, \textbf{A}_2, f ) \in \mathcal{J}(k)$. Then
\[
\emph{length} \big( \mathcal{O}^{\emph{sh}}_{\mathcal{J}_{x}} \big) = 2\emph{ length}\big( \mathcal{O}_{Y_1, x} \otimes_{\mathcal{O}_{X_N, x} } \mathcal{O}_{Y_2, x} \big).
\]
\end{lemma}
\begin{proof}
This is a consequence of the fact that $\text{deg}(\pi) = 1/2$.
\end{proof}

\begin{gevolg}\label{wschemes}
If $x = (\textbf{A}_1,\textbf{A}_2) \in (Y_1 \times Y_2)_{\vartheta}(k)$, then for any $\vartheta \in \mathcal{V}$ it holds that
\[
\emph{deg}(\mathcal{X}_{\vartheta}) = \frac{2}{w_1w_2} \sum_{\mathfrak{r} \subset \mathcal{O}_L} \emph{log}(|\mathbb{F}_{\mathfrak{r}}|) \sum_{ x \in (Y_1 \times Y_2)_{\vartheta}(k) } \emph{length}\big( \mathcal{O}_{Y_1, x} \otimes_{\mathcal{O}_{X_N, x} } \mathcal{O}_{Y_2, x} \big).
\]
\end{gevolg}
\begin{proof}
This is clear from the definition of the stack $\mathcal{X}_{\vartheta}$ and Lemma \ref{fstacks} above, with the single difference being the new sum not referencing the isomorphism $f$. However, it is easy to check that two triples $( \textbf{A}_1, \textbf{A}_2, f )$ and $( \textbf{A}'_1, \textbf{A}'_2, f' )$ are isomorphic as soon as all of the $\textbf{A}_i$ and $\textbf{A}_i'$ for $i \in \{ 1, 2 \}$ are isomorphic.
\end{proof}

As a result, to compute $\text{deg}( \mathcal{X}_{\vartheta} )$ for any given $\vartheta \in \mathcal{V}$, it suffices to study the numbers
\[
\text{length}\big( \mathcal{O}_{Y_1, x} \otimes_{\mathcal{O}_{X_N, x} } \mathcal{O}_{Y_2, x} \big)
\]
for $\mathfrak{r} \subset \mathcal{O}_L$ and $x= (\textbf{A}_1,\textbf{A}_2) \in (Y_1 \times Y_2)_{\vartheta}(k)$. The coarse moduli space $X_N$ is a genus 0 curve, thus allowing an isomorphism $j_N \colon X_N \to \mathbb{P}^1$. Morita proved in his master thesis \cite{Morita} that for any positive integer $N$, the Shimura curve $X_N$ is semistable and has good reduction at all the primes not dividing $N$.

Let $\mathfrak{r} \subset \mathcal{O}_L$ be a prime and let $\mathfrak{X}_{N,\mathfrak{r}}$ be a semistable model of $X_N$ over $\text{Spec}( \mathcal{O}_{L, \mathfrak{r}} )$. By the above, if $\mathfrak{r} \nmid N$, the completed local ring of any point on the special fibre is isomorphic to $\mathcal{O}_{L, \mathfrak{r}} \llbracket x \rrbracket$. If $\mathfrak{r} \mid N$, then because we chose a semistable model of $\mathfrak{X}_{N,r}$, at all the singular points on the special fibre the completed local ring is isomorphic to $\mathcal{O}_{L,\mathfrak{r}} \llbracket x,y \rrbracket / (xy - \varpi)$, where $\varpi$ denotes a uniformiser inside $\mathcal{O}_{L, \mathfrak{r}}$. However, because all CM-points we study are defined over the fields $H_i$ for $i \in \{ 1,2 \}$, both of which are unramified at $r$ by virtue of the assumption that $r$ be coprime to both $D_i$ for $i \in \{ 1,2 \}$, CM-points can never reduce to singular points on the special fibre on $\mathfrak{X}_{N,\mathfrak{r}}$. As such, the completed local rings at the reduction of a CM-point is always isomorphic to $\mathcal{O}_{L, \mathfrak{r}} \llbracket x \rrbracket$ as well.

As a result, for any $\mathfrak{r} \subset \mathcal{O}_L$ and $(\textbf{A}_1,\textbf{A}_2) \in (Y_1 \times Y_2)_{\vartheta}(k)$, we obtain two maps 
\[
\text{Spec}(\mathcal{O}_{L_{\mathfrak{r}}}) \to \text{Spec}(\mathcal{O}_{L_{\mathfrak{r}}}\llbracket x \rrbracket).
\]
which correspond to two ring maps
\[
\mathcal{O}_{L_{\mathfrak{r}}}\llbracket x \rrbracket \to \mathcal{O}_{L_{\mathfrak{r}}}.
\]
The following result is an easy exercise, the proof of which we opt to omit.
\begin{lemma}\label{tensorlem}
Let $R$ be a complete local ring and consider two maps $f_1,f_2 \colon R \llbracket x \rrbracket \to R$ defined by $f_1(x) = a$ and $f_2(x) = b$ for certain $a,b \in R$. Then
\[
R \otimes_{f_1, R \llbracket x \rrbracket, f_2} R \cong R/(a-b).
\]  
\end{lemma}
If we let $P_{\textbf{A}_i}$ denote the CM-point on $X_N$ defining the CM-false elliptic curve $\textbf{A}_i$, then the lemma above has the following immediate corollary.
\begin{gevolg}\label{ringlen}
For any $\mathfrak{r} \subset \mathcal{O}_L$ and $x = (\textbf{A}_1,\textbf{A}_2) \in (Y_1 \times Y_2)_{\vartheta}(k)$, it holds that
\[
\emph{length}\big( \mathcal{O}_{Y_1, x} \otimes_{\mathcal{O}_{\mathfrak{X}_N, x}} \mathcal{O}_{Y_2, x} \big) = v_{\mathfrak{r}}\big( j_{N,\mathfrak{r}}( P_{\textbf{A}_1} ) - j_{N,\mathfrak{r}}( P_{\textbf{A}_2} ) \big).
\]
\end{gevolg}
\begin{proof}
We apply Lemma \ref{tensorlem} to equate the left hand side to $\text{length}\left( \mathcal{O}_{L_{\mathfrak{r}}} / \big( j_{N,\mathfrak{r}}( P_{\textbf{A}_1} ) - j_{N,\mathfrak{r}}( P_{\textbf{A}_2} ) \big) \right)$, where $j_{N,\mathfrak{r}}$ is a suitably normalised scalar multiple of $j_N$; it is an easy exercise to check that this length is simply the $\mathfrak{r}$-adic valuation.
\end{proof}

\subsection{Proof of Theorem \ref{giamconj1}}

It is explained in Section 4.1 and 4.2 in \cite{anphil} how the group $\text{Pic}(K_i) \times W_N$ acts on the set of false elliptic curves with CM by $\mathcal{O}_i$. Key are Proposition 4.2.1, which states that this group action is simply transitive on the set $[ \mathcal{Y}_i( \mathbb{C} ) ]$, and Proposition 5.1.4, which states the same for $[ \mathcal{Y}_i( k ) ]$. It is important to know how the group actions on these embeddings relate to the reflex ideal $\mathfrak{a} \in \mathcal{I}$. 

\begin{lemma}\label{Galorient}
Let $(\textbf{A}_1, \textbf{A}_2)$ be a CM-pair inducing the morphism $\vartheta \in \mathcal{V}$. Then for any pair of ideals $([c_1], [c_2]) \in \emph{Pic}(K_1) \times \emph{Pic}(K_2)$, the CM-pair $( [c_1] \cdot \textbf{A}_1, [c_2] \cdot \textbf{A}_2 )$ also induces the map $\vartheta \in \mathcal{V}$.
\end{lemma}
\begin{proof}
This is immediate from the discussion in \cite{anphil} on page 39.
\end{proof}

\begin{lemma}\label{ALorient}
Let $(\textbf{A}_1, \textbf{A}_2)$ be a CM-pair inducing the reflex ideal $\mathfrak{a} = \mathfrak{p}_i\mathfrak{q}_j$. Then the CM-pairs $(w_p \cdot \textbf{A}_1, \textbf{A}_2)$ and $( \textbf{A}_1, w_p \cdot \textbf{A}_2)$ induce reflex ideal $\mathfrak{p}_k\mathfrak{q}_j$ with $k \neq i$, and the CM-pairs $(w_q \cdot \textbf{A}_1, \textbf{A}_2)$ and $( \textbf{A}_1, w_q \cdot \textbf{A}_2)$ induce reflex ideal $\mathfrak{p}_i\mathfrak{q}_l$ with $l \neq j$.
\end{lemma}
\begin{proof}
This is a consequence of the considerations in Section 5.2 in \cite{anphil}.
\end{proof}
\begin{gevolg}\label{embslemma}
For any given $\vartheta \in \mathcal{V}$, the space $(Y_1 \times Y_2)_{\vartheta}( k )$ is a principal homogeneous space for the action of $\emph{Pic}(K_1) \times \emph{Pic}(K_2)$. In addition, the set $\mathcal{V}$ itself is a principal homogenous space for the action of $W_N \times W_N$ and the set $\mathcal{I}$ is so for $W_N$.
\end{gevolg}
\begin{proof}
The first claim is a direct consequence of the results above, which combine to show that precisely the action of $\text{Pic}(K_1) \times \text{Pic}(K_2)$ leaves the morphism $\vartheta \in \mathcal{V}$ invariant. The other claims are easy to deduce from Lemma \ref{ALorient}.
\end{proof}
\begin{lemma}\label{wnconj}
For any point CM-point $P = P_{\textbf{A}_i}$, it holds that
\[
w_N(P) \in \emph{Pic}(K_i) \cdot \emph{Frob}_p(P).
\]
\end{lemma}
\begin{proof}
This is a consequence of the fact that $\text{Frob}_p$, acting nontrivially on $K_i$, also conjugates the reflex ideal to the other element in its $\text{Gal}(F / \mathbb{Q})$-orbit. This effect coincides with the action of $w_N$ on the reflex ideal and as such, these actions must agree up to an element from $\text{Pic}(K_i)$. 
\end{proof}
\begin{gevolg}\label{primewq}
For any CM-point $P = P_{\textbf{A}_i}$, it holds that
\[
P' \in \emph{Pic}(K_i) \cdot w_q(P).
\]
\end{gevolg}
\begin{proof}
Apply $w_p$ to both sides of Lemma \ref{wnconj} and recall that by definition, $P' = w_p ( \text{Frob}_p(P) )$. 
\end{proof}

\begin{proof}(of Theorem \ref{giamconj1}) Using the structure of $(Y_1 \times Y_2)_{\vartheta}(k)$ as principal homogeneous space, we allow ourselves to fix a CM-pair $(\textbf{A}_1, \textbf{A}_2)$ inducing some $\vartheta \in \mathcal{V}$.  Using Corollaries \ref{wschemes} and \ref{ringlen}, we now rearrange
\begin{align*}
\text{deg}( \mathcal{X}_{ \vartheta} ) &= \frac{2}{w_1w_2} \sum_{\mathfrak{r} \subset \mathcal{O}_L} \log(|\mathbb{F}_{\mathfrak{r}}|) \sum_{ \substack{ (\textbf{A}_1,\textbf{A}_2) \in \\ (Y_1 \times Y_2)_{\vartheta}(k)} } v_{\mathfrak{r}}\left( j_{N,\mathfrak{r}}( P_{\textbf{A}_1} ) - j_{N,\mathfrak{r}}( P_{\textbf{A}_2} ) \right) \\
 &= \frac{2}{w_1w_2} \sum_{\mathfrak{r} \subset \mathcal{O}_L} \log(|\mathbb{F}_{\mathfrak{r}}|) \sum_{([c_1],[c_2])} v_{\mathfrak{r}}\left( j_{N,\mathfrak{r}}( P_{[c_1] \cdot \textbf{A}_1} ) - j_{N,\mathfrak{r}}( P_{[c_2] \cdot \textbf{A}_2} ) \right).
\end{align*}
Now let $\mathcal{V}' = (W_q \times W_q) \cdot \vartheta \subset \mathcal{V}$ where $W_q = \{ 1, w_q \} \subset W_N$. We will study
\[
\sum_{ \vartheta \in \mathcal{V}' } \delta( \vartheta ) \text{deg}( \mathcal{X}_{ \vartheta} ).
\]
Making use of Corollary \ref{primewq} and exploiting the fact that we take an average over the product of the class groups to ignore the difference between $P'$ and $w_q(P)$, the contribution for $\mathfrak{r} \subset \mathcal{O}_L$ is precisely
\[
\sum_{([c_1],[c_2])} v_{\mathfrak{r}}\left( \frac{ j_{N,\mathfrak{r}}( P_{[c_1] \cdot \textbf{A}_1} ) - j_{N,\mathfrak{r}}( P_{[c_2] \cdot \textbf{A}_2} ) }{ j_{N,\mathfrak{r}} ( P'_{[c_1] \cdot \textbf{A}_1} ) - j_{N,\mathfrak{r}}( P_{[c_2] \cdot \textbf{A}_2} ) } \cdot \frac{ j_{N,\mathfrak{r}}( P'_{[c_1] \cdot \textbf{A}_1} ) - j_{N,\mathfrak{r}}( P'_{[c_2] \cdot \textbf{A}_2} ) }{ j_{N,\mathfrak{r}}( P_{[c_1] \cdot \textbf{A}_1} ) - j_{N,\mathfrak{r}}( P'_{[c_2] \cdot \textbf{A}_2} ) } \right).
\]
The cross ratio is independent of the choice of uniformising function $j_{N,\mathfrak{r}}$, so we may replace it by our original choice $j_N$ without changing the outcome. Now recall Shimura's reciprocity law, clearly explained on the first pages of \cite{shim67}, which states in effect that taking an average over the class group amounts to taking the norm of the algebraic integer $j_N( P )$ in the unramified field extension $H_i / K_i$. In other words,
\[
\prod_{[c_1],[c_2]} \big[ j_N( P_{[c_1] \cdot \textbf{A}_1} ), j_N( P'_{[c_1] \cdot \textbf{A}_1} ), j_N( P_{[c_2] \cdot \textbf{A}_2} ), j_N( P'_{[c_2] \cdot \textbf{A}_2}) \big]
\]
is equal to the norm $\text{Nm}^{H_1H_2}_{L} [ j_N( P_{\textbf{A}_1} ), j_N( P'_{\textbf{A}_1} ), j_N( P_{\textbf{A}_2} ), j_N( P'_{\textbf{A}_2}) ]$. We conclude that 
\begin{align*}
\sum_{ \vartheta \in \mathcal{V}' } \delta( \vartheta ) \text{deg}( \mathcal{X}_{ \vartheta} ) &= \frac{2}{w_1w_2} \sum_{\mathfrak{r} \subset \mathcal{O}_L} \log(|\mathbb{F}_{\mathfrak{r}}|) v_{\mathfrak{r}}\left( \text{Nm}^{H_1H_2}_{L} \big[ j_N( P_{\textbf{A}_1} ), j_N( P'_{\textbf{A}_1} ), j_N( P_{\textbf{A}_2} ), j_N( P'_{\textbf{A}_2}) \big] \right) \\
 &= \frac{2}{w_1w_2} \log \text{Nm}^{H_1H_2}_{\mathbb{Q}} \big[ j_N( P_{\textbf{A}_1} ), j_N( P'_{\textbf{A}_1} ), j_N( P_{\textbf{A}_2} ), j_N( P'_{\textbf{A}_2}) \big].
\end{align*}
Since $\mathcal{V}'$ contains four elements, two of each possible sign for the reflex ideal $\mathfrak{a}_{\theta}$, we may finally appeal to Corollary \ref{thm1cor} to conclude that
\[
\frac{2}{w_1w_2} \log \text{Nm}^{H_1H_2}_{\mathbb{Q}} \big[ j_N( P_{\textbf{A}_1} ), j_N( P'_{\textbf{A}_1} ), j_N( P_{\textbf{A}_2} ), j_N( P'_{\textbf{A}_2}) \big] = \prod_{\substack{ x^2 < D \\  x^2 \equiv D \mod 4N } } F\left( \frac{D - x^2}{4N} \right)^{\delta(x)};
\]
completing the proof of Theorem \ref{giamconj1}.
\end{proof}

\newpage

\section{An \texorpdfstring{$R=T$}{R=T} theorem}\label{modularity}

In this section we will study the theory of deformations of a rigidification of the Galois representation $\mathbbm{1} \oplus \chi$ attached to the following $p$-stabilisation of the Hilbert Eisenstein series $E_{1,\chi}^{(p)} \colonequals (1 - V_{\mathfrak{p}_1})(1 + V_{\mathfrak{p}_2})E_{1,\chi}$. The purpose of this rigidification is two-fold. First and foremost, it forces all endomorphisms of our representation to be scalar, resulting in a representable deformation functor. The second purpose is more subtle; rigidifying will cause the dimensions of the tangent spaces of our various deformation rings to drop by 1 compared to the naive unrigidified case due to the introduction of additional coboundaries. This is crucial in lining up the dimensions in order to prove our desired nearly ordinary modularity theorem. In this section, all homomorphisms and cocycles are assumed to be continuous. Similar arguments in different settings can be found in Pozzi's thesis \cite{Pozzi} and the works \cite{BDP, BD, BDS, BelChen}.

\subsection{Some Galois cohomology}

The following is key in enabling us to compute Galois cohomology groups for certain actions of absolute Galois groups on local fields; it can be found as Lemma 3.2 in \cite{DPV2}.
\begin{prop}\label{galhoms}
Let $H / F$ be finite Galois. Then there is an exact sequence of $\emph{Gal}(H/F)$-modules,
\begin{align*}
0 \to \emph{Hom}( G_H, \mathbb{Q}_p ) \to \prod_{ v \mid p } &\emph{Hom}( H_v^{\times}, \mathbb{Q}_p ) \to \emph{Hom}( \mathcal{O}_H[1/p]^{\times}, \mathbb{Q}_p ).
\end{align*}
In addition, this sequence is right exact if and only if Leopoldt's conjecture holds for $H$. 
\end{prop}
We omit the proof, but instead sketch how it is used to compute all cohomology groups that will be relevant for us. From now on, let $\mathbb{Q}_p(\chi)$ denote the $G_F$-module in which the action of $G_F$ on $\mathbb{Q}_p$ is through the character $\chi$. Further, write for simplicity $\mathbb{Q}_p( \mathbbm{1} ) \colonequals \mathbb{Q}_p$. 
\begin{lemma}\label{gfhomdim}
The space $\emph{Hom}( G_F, \mathbb{Q}_p )$ is 1-dimensional.
\end{lemma}
\begin{proof}
Since $F / \mathbb{Q}$ is abelian, Leopoldt's conjecture is known to be true and so the sequence from Proposition \ref{galhoms} is short exact for $H = F$. We may now count dimensions, using that
\[
F_{\mathfrak{p}_1}^{\times} \times F_{\mathfrak{p}_2}^{\times} \cong \mathbb{Q}_p^{\times} \times \mathbb{Q}_p^{\times} \quad \text{and} \quad \mathcal{O}_F[1/p]^{\times} / \{ \pm 1 \} \cong \mathbb{Z}^3.
\]
We conclude that $\text{dim }\text{Hom}( G_F, \mathbb{Q}_p ) = 4 - 3 = 1$, as claimed.
\end{proof}
For the sake of brevity, we will denote $G_{\mathfrak{p}_i} \colonequals G_{F_{\mathfrak{p}_i}}$ for $i = 1,2$. 
\begin{lemma}\label{gfH1dim}\label{gfpH1dim}
For $i \in \{ 1,2 \}$ the space $H^1( G_{\mathfrak{p}_i}, \mathbb{Q}_p(\chi) )$ is 1-dimensional. The restriction maps yield an isomorphism
\[
H^1( G_F, \mathbb{Q}_p(\chi) ) \xrightarrow{\sim} H^1( G_{\mathfrak{p}_1}, \mathbb{Q}_p(\chi) ) \oplus H^1( G_{\mathfrak{p}_2}, \mathbb{Q}_p(\chi) ).
\]
In particular, the space $H^1( G_F, \mathbb{Q}_p(\chi) )$ is 2-dimensional.
\end{lemma}
\begin{proof}
We follow Lemma 3.3 in \cite{DPV2}. First, we note that restriction to $G_L$ gives an isomorphism
\[
H^1(G_F, \mathbb{Q}_p(\chi)) \cong \text{Hom}( G_L,  \mathbb{Q}_p(\chi) )^{\text{Gal}(L/F)} =  \text{Hom}( G_L,  \mathbb{Q}_p)^{\chi},
\]
where the final group denotes the $\chi$-eigenspace. Indeed, this follows from the inflation-restriction sequence and the easy computation that $H^1( \text{Gal}(L/F), \mathbb{Q}_p(\chi)) = H^2( \text{Gal}(L/F), \mathbb{Q}_p(\chi)) = 0$. Since $L / \mathbb{Q}$ is abelian, the sequence in Proposition \ref{galhoms} is exact as before. We obtain an isomorphism
\[
\text{Hom}( G_L,  \mathbb{Q}_p)^{\chi} \cong \text{Hom}( L_{\mathfrak{p}_1}^{\times}, \mathbb{Q}_p )^{\chi} \times \text{Hom}( L_{\mathfrak{p}_2}^{\times}, \mathbb{Q}_p )^{\chi},
\]
as by the Galois-equivariant version of Dirichlet's Unit Theorem, it holds that $( \mathcal{O}_L^{\times}[1/p] \otimes \mathbb{Q}_p )^{\chi} = 0$, which uses that $\chi$ is totally odd and $\chi( \mathfrak{p}_i ) = -1$ for $i \in \{ 1,2 \}$. Similarly, one can show that for $i \in \{ 1,2 \}$,
\[
H^1( G_{\mathfrak{p}_i}, \mathbb{Q}_p(\chi) ) \cong \text{Hom}( G_{L_{\mathfrak{p}_i}},  \mathbb{Q}_p)^{\chi} \cong \text{Hom}( L_{\mathfrak{p}_i}^{\times},  \mathbb{Q}_p)^{\chi},
\]
proving the claimed isomorphism. Finally, the spaces $\text{Hom}( L_{\mathfrak{p}_i}^{\times},  \mathbb{Q}_p)$ for $i \in \{ 1, 2 \}$ are easily seen to be 3-dimensional and spanned by $\text{ord}_{\mathfrak{p}_i}$, a $p$-adic logarithm and its Galois twist. Since $\chi$ restricts non-trivially to the decomposition group, the $\chi$-eigenspace is quickly seen to be 1-dimensional, spanned by the difference of the two logarithms. This completes the proof of the lemma.
\end{proof}
Finally, we will need the following results in order to compute tangent spaces later on.
\begin{lemma}\label{H2vanish}
It holds that
\[
H^2( G_F, \mathbb{Q}_p ) = 0 \quad \text{and} \quad H^2( G_F, \mathbb{Q}_p(\chi) ) = 0.
\]
\end{lemma}
\begin{proof}
We use the global Euler characteristic formula. Since $F$ has two real places, we compute that
\begin{align*}
\text{dim } H^2( G_F, \mathbb{Q}_p ) &= \text{dim } H^1( G_F, \mathbb{Q}_p ) - \text{dim } H^0( G_F, \mathbb{Q}_p ) + 2 \text{ dim } H^0( G_{\mathbb{R}}, \mathbb{Q}_p ) - 2 \text{ dim }(\mathbb{Q}_p) \\
 &= 1 - 1 + 2 - 2 = 0,
\end{align*}
where we used Proposition \ref{gfhomdim}. Similarly, using Proposition \ref{gfH1dim}, we compute that
\begin{align*}
\text{dim } H^2( G_F, \mathbb{Q}_p(\chi) ) &= \text{dim } H^1( G_F, \mathbb{Q}_p(\chi) ) - \text{dim } H^0( G_F, \mathbb{Q}_p(\chi) ) + 2 \text{ dim } H^0( G_{\mathbb{R}}, \mathbb{Q}_p(\chi) ) \\ &- 2 \text{ dim }(\mathbb{Q}_p(\chi))
 = 2 - 0 + 0 - 2 = 0,
\end{align*}
because complex conjugation in $\text{Gal}( \mathbb{C} / \mathbb{R} )$ acts non-trivially through $\chi$.
\end{proof}

\subsection{Deformation functors and representability}
The following lemma follows from a direct calculation.
\begin{lemma}\label{rhoetaired}
For any $\eta \in Z^1( G_F, \mathbb{Q}_p(\chi) )$, we have a Galois representation
\[
\rho_{\eta} \colon G_F \to \emph{GL}_2( \mathbb{Q}_p ) \colon \tau \mapsto \begin{pmatrix} 1 & \chi(\tau)\eta(\tau) \\ 0 & \chi(\tau) \end{pmatrix}.
\]
In addition, it has no non-scalar endomorphisms if and only if $\eta \neq 0 \in H^1( G_F, \mathbb{Q}_p(\chi) )$.
\end{lemma}
\begin{definition}
Let $\mathcal{C}_{\mathbb{Q}_p}$ denote the category of local complete Noetherian $\mathbb{Q}_p$-algebras with residue field $\mathbb{Q}_p$. Given any object $(A, \mathfrak{m}_A)$ of $\mathcal{C}_{\mathbb{Q}_p}$, a \emph{lift} of $\rho_{\eta}$ to $A$ is a representation $\rho \colon G_F \to \text{GL}_2( A )$ that reduces to $\rho_{\eta}$ after composing with the natural map $\text{GL}_2( A ) \to \text{GL}_2( \mathbb{Q}_p )$ induced by the natural map $A \mapsto A / \mathfrak{m}_A \cong \mathbb{Q}_p$. We say that two lifts are equivalent if they are conjugate by a matrix in the kernel of the map $\text{GL}_2( A ) \to \text{GL}_2( \mathbb{Q}_p )$. A \emph{deformation} of $\rho_{\eta}$ to $A$ is an equivalence class of lifts of $\rho_{\eta}$ to $A$. We define the functor $D_{\rho_{\eta}} \colon \mathcal{C}_{\mathbb{Q}_p} \to \textbf{Set}$ by sending any $(A, \mathfrak{m}_A) \in \mathcal{C}_{\mathbb{Q}_p}$ to the set of equivalence classes of deformations of $\rho_{\eta}$ to $A$. 
\end{definition}
\begin{prop}
If $\eta \neq 0 \in H^1( G_F, \mathbb{Q}_p(\chi) )$, the functor $D_{\rho_{\eta}}$ is represented by a ring $R_{\rho_{\eta}}$.
\end{prop}
\begin{proof}
Using the same argument as in Proposition 5 in \cite{mazurdef}, we reduce to showing the non-existence of non-scalar endomorphisms and the finite dimensionality of the tangent space. The former is taken care of by Lemma \ref{rhoetaired}, and for the latter we may bound the dimension of the tangent space $H^1( G_F, \text{ad}(\rho_{\eta}) )$ by the dimension of the tangent space of its semisimplification $H^1( G_F, \text{ad}(\rho_0) )$. By standard arguments, this space decomposes as $\text{Hom}( G_F, \mathbb{Q}_p )^2 \oplus H^1( G_F, \mathbb{Q}_p( \chi ) )^2$; whence by Lemmas \ref{gfhomdim} and \ref{gfH1dim}, its dimension is $6 < \infty$. 
\end{proof}
From now on, we assume that $\eta \neq 0 \in H^1( G_F, \mathbb{Q}_p(\chi) )$ to ensure representability.
\begin{definition}
Consider triples $(\rho, L_1, L_2)$ where $\rho$ is a lift of $\rho_{\eta}$ to some $(A, \mathfrak{m}_A) \in \mathcal{C}_{\mathbb{Q}_p}$ and $L_1$ and $L_2$ are free direct summands of $A^2$ such that $L_1$ lifts the line $\langle e_1 \rangle$ and $L_2$ lifts the line $\langle e_2 \rangle$. We say two such triples are equivalent if for some $g \in \text{ker}(\text{GL}_2( A ) \to \text{GL}_2( \mathbb{Q}_p ))$ it holds that $\rho = g\rho'g^{-1}$ and $L_1 = gL_1'$ and $L_2 = gL_2'$. We let $D_{\rho_{\eta}}^{\text{fil}}$ be the functor sending an object $(A, \mathfrak{m}_A)$ to the set of equivalence classes of such triples $(\rho, L_1, L_2)$ as defined above.
\end{definition}
\begin{prop}\label{fildefring}
The functor $D_{\rho_{\eta}}^{\emph{fil}}$ is represented by the ring $R_{\rho_{\eta}}\llbracket X, Y \rrbracket$.
\end{prop}
\begin{proof}
We define a bijection
\[
\text{Hom}( R_{\rho_{\eta}}\llbracket X,Y \rrbracket, A ) \to D_{\rho_{\eta}}^{\text{fil}}(A)
\]
by sending some map $f \colon R_{\rho_{\eta}}\llbracket X,Y \rrbracket \to A$ to the representation 
\[
G_F \to \text{GL}_2(R_{\rho_{\eta}}) \to \text{GL}_2( R_{\rho_{\eta}}\llbracket X,Y \rrbracket ) \to \text{GL}_2(A)
\]
induced by the universal representation and the map $f$ itself, together with the lines $L_{f,1} = \langle e_1 + f(X)e_2 \rangle \subset A^2$ and $L_{f,2} = \langle e_2 + f(Y)e_1 \rangle \subset A^2$. As $f$ is a morphism of local rings, $f(X), f(Y) \in \mathfrak{m}_A$ and thus the map is well defined. Since every line lifing $\langle e_i \rangle$ to $A$ for some $i \in \{ 1,2 \}$ is of that form and we may choose the images of $X$ and $Y$ freely in $\mathfrak{m}_A$, surjectivity is obvious. We sketch the proof of injectivity; more details can be found in Lemma 1.3.2 in \cite{Pozzi}. If two morphisms $f, f' \colon R_{\rho_{\eta}}\llbracket X, Y \rrbracket \to A$ have the same image under the above association, they give rise to the same deformation. But then the universal property of $R_{\rho_{\eta}}$ yields readily that $f$ and $f'$ must agree when restricted to $R_{\rho_{\eta}}$. This means that the matrix $g \in \text{ker}(\text{GL}_2( A ) \to \text{GL}_2( \mathbb{Q}_p ))$ intertwines the lift induced by both $f$ and $f'$, which forces $g$ to be scalar. This readily yields $f(X) = f'(X)$ and $f(Y) = f'(Y)$ and as such $f = f'$ on all of $R_{\rho_{\eta}}\llbracket X,Y \rrbracket$. 
\end{proof}
From now on, we must require that $\eta|_{G_{\mathfrak{p}_2}} = 0$. The reason as to why is immediate from the following definition; the line $\langle e_2 \rangle$ is only fixed by $G_{\mathfrak{p}_2}$ by $\rho_{\eta}$ if this condition on $\eta$ is satisfied. Note that this fixes $\eta$ uniquely up to a scalar, as can be seen from Lemma \ref{gfpH1dim}.
\begin{definition}
Let $D_{\rho_{\eta}}^{\text{no}} \colon \mathcal{C}_{\mathbb{Q}_p} \to \textbf{Set}$ be the subfunctor of $D_{\rho_{\eta}}^{\text{fil}}$ sending an object $(A, \mathfrak{m}_A) \in \mathcal{C}_{\mathbb{Q}_p}$ to the equivalence class of triples $(\rho, L_1, L_2)$ as above with the additional properties that the line $L_1$ is $G_{\mathfrak{p}_1}$-stable and the line $L_2$ is $G_{\mathfrak{p}_2}$-stable. We call such deformations \emph{nearly ordinary}.
\end{definition}
By requiring the two lines $L_1$ and $L_2$ to lift the two distinct lines $\langle e_1 \rangle$ and $\langle e_2 \rangle$ respectively in the definition above, we ensure that the two quotient characters on the spaces $A^2 / L_i$ will lift the two distinct characters $\chi$ and $\mathbbm{1}$. This corresponds to the particular choice of $p$-stabilisation $E_{1,\chi}^{(p)} \colonequals (1 - V_{\mathfrak{p}_1})(1 + V_{\mathfrak{p}_2})E_{1,\chi}$ with two distinct signs that will be at the core of our arguments later.
\begin{prop}\label{nodefring}
The functor $D_{\rho_{\eta}}^{\emph{no}}$ is represented by a universal deformation ring $R_{\rho_{\eta}}^{\emph{no}}$.
\end{prop}
\begin{proof}
We will find an ideal $I \subset R_{\rho_{\eta}}\llbracket X \rrbracket$ such that in the bijection
\[
\text{Hom}( R_{\rho_{\eta}}\llbracket X,Y \rrbracket, A ) \to D_{\rho_{\eta}}^{\text{fil}}(A),
\]
the image of an element on the left is contained in the subset $D_{\rho_{\eta}}^{\text{no}}(A)$ if and only if it factors through $R_{\rho_{\eta}}^0\llbracket X,Y \rrbracket / I$. This would yield a bijection
\[
\text{Hom}( R_{\rho_{\eta}}\llbracket X,Y \rrbracket / I, A ) \to D_{\rho_{\eta}}^{\text{no}}(A),
\]
establishing the desired conclusion $R_{\rho_{\eta}}^{\text{no}} \cong R_{\rho_{\eta}}\llbracket X,Y \rrbracket / I$. It remains to identify $I$.  Let us consider a representative for the universal deformation
\[
\rho^{\text{univ}} = \begin{pmatrix} \alpha & \beta \\ \gamma & \delta \end{pmatrix} 
\]
and investigate when its universal line $L_1^{\text{univ}} = \langle e_1 + X e_2 \rangle$ is stable under the action of $G_{\mathfrak{p}_1}$. By changing bases, this happens precisely when
\[
\begin{pmatrix} 1 & 0 \\ -X & 1 \end{pmatrix} \begin{pmatrix} \alpha & \beta \\ \gamma & \delta \end{pmatrix} \begin{pmatrix} 1 & 0 \\ X & 1 \end{pmatrix} = \begin{pmatrix} \alpha + \beta X & \beta \\ \gamma + (\delta-\alpha)X - \beta X^2 & \delta-\beta X \end{pmatrix}
\]
fixes the line $\langle e_1 \rangle$ on $G_{\mathfrak{p}_1}$. This is easy to read off; it happens precisely when
\[
\gamma(\sigma) + (\delta(\sigma) - \alpha(\sigma))X - \beta(\sigma)X^2\quad \text{vanishes for all } \sigma \in G_{\mathfrak{p}_1}.
\]
Let $I_1 \subset R_{\rho_{\eta}}^0\llbracket X \rrbracket$ be the ideal generated by all the elements above and completely similarly define $I_2$. Then $I = I_1 + I_2$ is easily seen to be the desired ideal.
\end{proof}

\subsection{Computing tangent spaces}

Choosing a basis of $\mathbb{Q}_p^2$, we may identify the adjoint representation
\[
\text{Ad}( \rho_{\eta} ) \cong M_2( \mathbb{Q}_p ),
\]
on which the action is given for $g \in G_F$ by $g \cdot M = \rho_{\eta}(g)^{-1} M \rho_{\eta}(g)$. 
\begin{lemma}\label{phi1def}
There is a well-defined map of $G_F$-modules
\[
\varphi_1 \colon \emph{Ad}( \rho ) \to \mathbb{Q}_p( \chi ) \colon \begin{pmatrix} x & y \\ z & w \end{pmatrix} \mapsto z.
\]
\end{lemma}
\begin{proof}
This follows from the computation of the matrix product
\begin{equation}
\begin{pmatrix} 1 & -\eta \\ 0 & \chi \end{pmatrix} \begin{pmatrix} x & y \\ z & w \end{pmatrix} \begin{pmatrix} 1 & \chi\eta \\ 0 & \chi \end{pmatrix} = \begin{pmatrix} x - z \eta & x\chi\eta + y\chi - z\chi\eta^2 - w\chi\eta \\ z \chi & z\eta + w \end{pmatrix}, \label{mateq}
\end{equation}
which shows that the $G_F$-action on the bottom-right entry is through multiplication by $\chi$.
\end{proof}
Now let $W_1 = \text{ker}( \varphi_ 1)$; in other words, we have a short exact sequence of $G_F$-modules
\[
0 \to W_1 \to \text{Ad}( \rho ) \to \mathbb{Q}_p( \chi ) \to 0.
\]
\begin{lemma}\label{xwmap}
There is a well-defined map of $G_F$-modules
\[
\varphi_2 \colon W_1 \to \mathbb{Q}_p \oplus \mathbb{Q}_p \colon \begin{pmatrix} x & y \\ 0 & w \end{pmatrix} \mapsto (x,w).
\]
\end{lemma}
\begin{proof}
This is immediate from substituting $z = 0$ in Equation \ref{mateq}.
\end{proof}
We now define $W_2 = \text{ker}(\varphi_2)$, so that we have a short exact sequence
\[
0 \to W_2 \to W_1 \to \mathbb{Q}_p \oplus \mathbb{Q}_p \to 0.
\]
The following result completes the filtration and allows us to start computing cohomology groups using long exact sequences.
\begin{lemma}\label{w2struc}
There is a well-defined isomorphism of $G_F$-modules given by
\[
W_2 \xrightarrow{\sim} \mathbb{Q}_p( \chi ) \colon \begin{pmatrix} 0 & y \\ 0 & 0 \end{pmatrix} \mapsto y.
\]
\end{lemma}
\begin{proof}
This is immediate from substituting $x = z = w = 0$ in Equation \ref{mateq}.
\end{proof}
\begin{prop}\label{W1coh}
The group $H^1( G_F, W_1 )$ is 3-dimensional. Further, $H^2( G_F, W_1 ) = 0$. 
\end{prop}
\begin{proof}
We consider the long exact sequence associated with the short exact sequence defining $W_2$. It is clear that $H^0( G_F, \mathbb{Q}_p( \chi ) ) = 0$ and $H^0( G_F, \mathbb{Q}_p \oplus \mathbb{Q}_p ) = \mathbb{Q}_p \oplus \mathbb{Q}_p$. Further, one may observe that
\[
H^0( G_F, W_1 ) = W_1^{G_F} = \big\{ M \in W_1 \mid \rho^{-1}M\rho = M \big\} = \langle \text{id} \rangle \cong \mathbb{Q}_p,
\]
where we used Lemma \ref{rhoetaired}. Finally, we recall Lemma \ref{H2vanish}, which states that $H^2( G_F, \mathbb{Q}_p( \chi ) ) = 0$. Combining all of this with Lemma \ref{w2struc}, the long exact sequence becomes
\[
0 \to \mathbb{Q}_p \to \mathbb{Q}_p \oplus \mathbb{Q}_p \to H^1( G_F, \mathbb{Q}_p( \chi ) ) \to H^1( G_F, W_1 ) \to H^1( G_F, \mathbb{Q}_p \oplus \mathbb{Q}_p ) \to 0.
\]
Applying Lemmas \ref{gfhomdim} and \ref{gfH1dim}, we conclude that
\[
\text{dim } H^1( G_F, W_1 ) + 2 = 1 + 2 + 2;
\]
completing the proof of the first claim. For the second, we look slightly further along in the long exact sequence and use Lemma \ref{H2vanish} to find $H^2( G_F, W_1 )$ in between two zeroes. 
\end{proof}
\begin{stelling}\label{tanspdim}
The tangent space $t_{\rho_{\eta}}$ is 5-dimensional.
\end{stelling}
\begin{proof}
Using the well-known isomorphism $t_{\rho_{\eta}} \cong H^1( G_F, \text{Ad}(\rho) )$, we reduce to computing the dimension of the latter cohomology group. We now use the long exact sequence associated with the short exact sequence defining $W_1$. Recalling that $H^0( G_F, \mathbb{Q}_p( \chi ) ) = 0$ and $H^2( G_F, W_1 ) = 0$ by Proposition \ref{W1coh} above, we conclude that part of this sequence reads
\[
0 \to H^1( G_F, W_1 ) \to H^1( G_F, \text{Ad}(\rho) ) \to H^1( G_F, \mathbb{Q}_p( \chi ) ) \to 0.
\]
In particular, appealing to Lemma \ref{gfH1dim} and again Proposition \ref{W1coh}, we find that
\[
\text{dim }H^1( G_F, \text{Ad}(\rho) ) = \text{dim }H^1( G_F, W_1 ) + \text{dim }H^1( G_F, \mathbb{Q}_p( \chi ) ) = 3 + 2 = 5,
\]
completing the proof.
\end{proof}
We now compute the dimension of the tangent space to the nearly ordinary deformation functor.
\begin{lemma}\label{filtang}
There is an isomorphism of  $\mathbb{Q}_p$-vector spaces between $t_{\rho_{\eta}}^{\emph{fil}}$ and $H^1( G_F, \emph{Ad}(\rho_{\eta} ) ) \oplus \mathbb{Q}_p^2$.
\end{lemma}
\begin{proof}
By definition and using Proposition \ref{fildefring}, we have
\[
t_{\rho_{\eta}}^{\text{fil}} = \text{Hom}( R_{\rho_{\eta}}^{\text{fil}}, \mathbb{Q}_p[\epsilon] ) = \text{Hom}( R_{\rho_{\eta}} \llbracket X,Y \rrbracket, \mathbb{Q}_p[\epsilon] ) \cong \text{Hom}( R_{\rho_{\eta}}, \mathbb{Q}_p[\epsilon] ) \oplus \mathbb{Q}_p^2;
\]
this final isomorphism comes from the observation that we may choose the images of $X$ and $Y$ arbitrarily and independently in the maximal ideal $\epsilon\mathbb{Q}_p[\epsilon] \cong \mathbb{Q}_p$.
\end{proof}
 
\begin{prop}\label{nocond}
A triple $(\Theta, \lambda_1, \lambda_2) \in H^1( G_F, \emph{Ad}(\rho_{\eta}) ) \oplus \mathbb{Q}_p^2$ corresponds to a nearly ordinary deformation of $\rho_{\eta}$ if and only if
\[
c|_{G_{\mathfrak{p}_1}} = \lambda_1 (1 - \chi) \quad \text{and} \quad b|_{G_{\mathfrak{p}_2}} = \lambda_2 ( \chi - 1 ).
\]
\end{prop}
\begin{proof}
By definition and using Proposition \ref{nodefring}, we have
\[
t_{\rho_{\eta}}^{\text{no}} = \text{Hom}( R_{\rho_{\eta}}^{\text{no}}, \mathbb{Q}_p[\epsilon] ) = \text{Hom}( R_{\rho_{\eta}} \llbracket X,Y \rrbracket / I, \mathbb{Q}_p[\epsilon] ),
\]
where $I = I_1 + I_2$ as in the proof of Proposition \ref{nodefring}. We retain its notation and consider $\varphi \in\text{Hom}( R_{\rho_{\eta}}^0 \llbracket X,Y \rrbracket, \mathbb{Q}_p[\epsilon] )$. If we write
\[
\Theta = \begin{pmatrix} a & b \\ c & d \end{pmatrix},
\]
then by Lemma \ref{filtang}, there is a unique triple $(\Theta, \lambda_1, \lambda_2) \in H^1( G_F, \text{End}^0(V) ) \oplus \mathbb{Q}_p^2$ such that
\[
\varphi \begin{pmatrix} \alpha & \beta \\ \gamma & \delta \end{pmatrix} = ( 1 + \epsilon \Theta)\rho_{\eta} = \begin{pmatrix} 1 + a\epsilon & \chi\eta + \chi[\eta a + b]\epsilon \\ c\epsilon & \chi + \chi[\eta c + d]\epsilon \end{pmatrix}, \quad \varphi(X) = \lambda_1 \epsilon, \quad \varphi(Y) = \lambda_2 \epsilon.
\]
We can now compute the constraints posed by the vanishing on $I$ to be
\[
c = \lambda_1 (1 - \chi) \quad \text{on $G_{\mathfrak{p}_1}$} \quad \text{and} \quad \chi\eta + \chi[\eta a + b]\epsilon + \lambda_2 ( 1 - \chi )\epsilon = 0 \quad \text{on $G_{\mathfrak{p}_2}$}.
\]
However, $\eta$ is chosen to be trivial on $G_{\mathfrak{p}_2}$. As such, we obtain the proposition.
\end{proof}
\begin{lemma}\label{nomapdef}
There is a well-defined homomorphism
\[
f \colon H^1( G_F, \emph{Ad}(\rho_{\eta}) ) \to H^1( G_{\mathfrak{p}_1}, \mathbb{Q}_p(\chi)) \oplus H^1( G_{\mathfrak{p}_2}, \mathbb{Q}_p(\chi))
\]
explicitly given by, adopting the usual notation for the components of $\Theta$, 
\[
\Theta \mapsto \big( c|_{G_{\mathfrak{p}_1}}, b|_{G_{\mathfrak{p}_2}} \big).
\]
\end{lemma}
\begin{proof}
It is easy to see that the map $H^1( G_F, \text{Ad}(\rho) ) \to H^1( G_F, \mathbb{Q}_p(\chi) )$ in the proof of Theorem \ref{tanspdim} is given by $\Theta \mapsto c \in H^1( G_F, \mathbb{Q}_p( \chi ) )$ so we may compose this map with the restriction $G_F \to G_{\mathfrak{p}_1}$. Further, since $\eta$ vanishes on $G_{\mathfrak{p}_2}$, it holds that $\rho_{\eta}|_{G_{\mathfrak{p}_2}} = \mathbbm{1} \oplus \chi$. This readily implies that also $b|_{G_{\mathfrak{p}_2}} \in H^1( G_{\mathfrak{p}_2}, \mathbb{Q}_p( \chi ) )$ describes a well-defined cohomology class.
\end{proof}
\begin{prop}\label{notanker}
There is an isomorphism of $\mathbb{Q}_p$-vector spaces between $t_{\rho_{\eta}}^{\emph{no}}$ and $\emph{ker}(f)$.
\end{prop}
\begin{proof}
Given $\Theta \in \text{ker}(f)$, we may construct a nearly ordinary triple $( \Theta, \lambda_1, \lambda_2)$ using the map
\[
\Theta \mapsto \big( \Theta, c(\text{Frob}_{\mathfrak{p}_1})/2, b(\text{Frob}_{\mathfrak{p}_2})/2 \big).
\]
This is in fact nearly ordinary, because by virtue of $\Theta$ being in the kernel of $f$, the cocycles $c|_{G_{\mathfrak{p}_1}}$ and $b|_{G_{\mathfrak{p}_2}}$ are coboundaries and as such, are given by $c|_{G_{\mathfrak{p}_1}} = \mu_1( 1 - \chi)$ and $b|_{G_{\mathfrak{p}_2}} = \mu_2 ( \chi - 1 )$ for certain $\mu_1, \mu_2 \in \mathbb{Q}_p$. Using that $\chi(\text{Frob}_{\mathfrak{p}_i}) = -1$ for $i \in \{ 1,2 \}$, evaluating yields that $\mu_1 = c(\text{Frob}_{\mathfrak{p}_1}) / 2$ and $\mu_2 = b(\text{Frob}_{\mathfrak{p}_2}) / 2$ are uniquely determined. Comparing with Proposition \ref{nocond}, this shows our claim. Conversely, to any nearly ordinary triple we associate its first component, since by Proposition \ref{nocond}, for nearly ordinary triples, $c|_{G_{\mathfrak{p}_1}}$ and $b|_{G_{\mathfrak{p}_2}}$ must be coboundaries. Since these operations are evidently inverse, this establishes the proposition.
\end{proof}
 
\begin{gevolg}\label{notandim}
The tangent space $t_{\rho_{\eta}}^{\emph{no}}$ is 3-dimensional.
\end{gevolg}
\begin{proof}
We claim that the map $f$ from Lemma \ref{nomapdef} is surjective. Indeed, the sequence in Theorem \ref{tanspdim} in combination with Lemma \ref{gfpH1dim} shows that the map onto the first factor is surjective. Further, the submodule $W_2$ of $\text{Ad}(\rho_{\eta})$ surjects onto the second factor by Lemma \ref{w2struc} while being identically zero on the first; these two observations imply surjectivity. Using Theorem \ref{tanspdim} and two applications of Lemma \ref{gfpH1dim}, we conclude that $\text{dim } t_{\rho_{\eta}}^{\text{no}} = 5 - 2 = 3$; precisely as claimed.
\end{proof}

\subsection{A lift to \texorpdfstring{$\mathbb{T}$}{T}}

Let $\mathbb{T}^{\text{no}}$ denote Hida's nearly ordinary cuspidal Hecke algebra as defined in \cite{JAMI} and discussed in Section 3 of \cite{DPV2}. Let $E_{1, \chi}$ denote the parallel weight $(1,1)$-Hilbert Eisenstein series and further we let $$E_{1,\chi}^{(p)} \colonequals (1 - V_{\mathfrak{p}_1})(1 + V_{\mathfrak{p}_2})E_{1,\chi}$$ be one of the $p$-stabilisations with opposite choices of signs. This is a $p$-adic cusp form and as such defines a morphism $\mathbb{T}^{\text{no}} \to \mathbb{Q}_p$ by sending a Hecke operator to its $f := E_{1,\chi}^{(p)}$-eigenvalue. Let $\mathbb{T}$ be the nilreduction of the completion of the localisation of $\mathbb{T}^{\text{no}}$ at the prime ideal $\mathfrak{m}_f$ given by the kernel of the morphism above. Let $\mathbb{K}$ be its ring of fractions, which is a product of fields. Then Hida proved the following in \cite{JAMI}.

\newpage

\begin{stelling}\label{heckealgrep}
There exists a unique semisimple Galois representation $\pi \colon G_F \to \emph{GL}_2( \mathbb{K} )$ with the following properties:
\begin{itemize}
\item $\pi$ is continuous, odd and unramified outside $p$;
\item For each prime $\mathfrak{l} \nmid p$, it holds that
\[
\det\big( 1 - \pi( \emph{Frob}_{\mathfrak{l}} ) X \big) = 1 - T_{\mathfrak{l}}X + \langle \mathfrak{l} \rangle \emph{Nm}(\mathfrak{l}) X^2;
\]
\item For $i \in \{ 1,2 \}$ there exist characters $\epsilon_i, \delta_i \colon G_{\mathfrak{p}_i} \to \mathbb{T}^{\times}$ such that, up to equivalence, when restricted to $G_{\mathfrak{p}_i}$, the representation $\pi$ is of the form
\[
\pi( \sigma) = \begin{pmatrix} \epsilon_i(\sigma) & * \\ 0 & \delta_i(\sigma) \end{pmatrix} \quad \emph{for} \quad \sigma \in G_{\mathfrak{p}_i}.
\]
\item If we identify $G_{\mathfrak{p}_i}^{\emph{ab}} \cong F_{\mathfrak{p}_i}^{\times}$, then we have the identity $\delta_i( x ) = U_x$ for all $x \in F_{\mathfrak{p}_i}^{\times}$.
\end{itemize}
\end{stelling}
We refine this representation in two ways using the nowadays standard technique. We must find a stable lattice inside $\mathbb{K}^2$ so that we obtain a representation $G_F \to \text{GL}_2( \mathbb{T} )$ instead, which we may then reduce modulo its maximal ideal $\mathfrak{m}_f$. Secondly, we must insist that this reduction precisely equals $\rho_{\eta}$ in order to obtain a deformation of $\rho_{\eta}$. Let us achieve these two results in succession.
\begin{lemma}\label{generalbasis}
There exist an element $\gamma \in G_F \setminus G_L$ and a basis of $\{ e_1, e_2 \}$ of $\mathbb{K}^2$ such that
\[
\pi( \gamma ) = \begin{pmatrix} \lambda_1 & 0 \\ 0 & \lambda_2 \end{pmatrix},
\]
where $\lambda_1 \equiv 1 \mod \mathfrak{m}_f$ and $\lambda_2 \equiv -1 \mod \mathfrak{m}_f$. In addition, the unique fixed lines fixed by the subgroups $G_{\mathfrak{p}_i}$, for $i \in \{ 1, 2 \}$ can be written as $\langle e_1 + y_i e_2 \rangle$ where $y_i \in \mathbb{K}^{\times}$.
\end{lemma}
\begin{proof}
This is the content of Lemma 4.3 and Lemma 4.6 in \cite{onGSconj}.
\end{proof}
In any basis as in the lemma above, write
\[
\pi(\sigma) = \begin{pmatrix} a(\sigma) & b(\sigma) \\ c(\sigma) & d(\sigma) \end{pmatrix} 
\]
for certain functions $a, b, c, d \colon G_F \to \mathbb{K}$. We proceed to analyse these functions.
\begin{lemma}\label{admodm}
The functions $a$ and $d$ are $\mathbb{T}$-valued. In fact, $a \equiv 1 \mod \mathfrak{m}_f$ and $d \equiv \chi \mod \mathfrak{m}_f$.
\end{lemma}
\begin{proof}
For any prime $\mathfrak{l} \nmid p$, using Theorem \ref{heckealgrep}, we have that 
\[
a(\text{Frob}_{\mathfrak{l}}) + d(\text{Frob}_{\mathfrak{l}}) = \text{Tr}( \pi( \text{Frob}_{\mathfrak{l}} ) ) = T_{\mathfrak{l}} \in \mathbb{T}.
\]
In other words, the continuous map $\text{Tr}(\pi) \colon G_F \to \mathbb{K}$ takes on integral values for every element $\text{Frob}_{\mathfrak{l}}$ for $\mathfrak{l} \nmid p$. By Chebotarev's Density Theorem, the result extends to all of $G_F$. Next note that
\[
\lambda_1 a(\text{Frob}_{\mathfrak{l}}) + \lambda_2 d(\text{Frob}_{\mathfrak{l}}) = \text{Tr}( \pi( \gamma \text{Frob}_{\mathfrak{l}} ) ) \in \mathbb{T}.
\]
Combining these two expressions and using that $\lambda_1 - \lambda_2 \in \mathbb{T}^{\times}$ then yields that $a$ and $d$ both must have integral image themselves. Finally, by definition of $f = E_{1,\chi}^{(p)}$ we have $T_{\mathfrak{l}} \equiv 1 + \chi( \mathfrak{l} ) \mod \mathfrak{m}_f$. Again, by continuity, this implies that for any $\sigma \in G_F$, it holds that $a(\sigma) + d(\sigma) \equiv 1 + \chi(\sigma) \mod \mathfrak{m}_f$. Again considering $\gamma \sigma$, we also obtain $a(\sigma) - d(\sigma) \equiv 1 - \chi(\sigma) \mod \mathfrak{m}_f$. Combining these completes the proof.
\end{proof}
\newpage
\begin{lemma}\label{bcinm}
For any $\sigma, \tau \in G_F$, it holds that $b(\sigma)c(\tau) \in \mathfrak{m}_f$.
\end{lemma}
\begin{proof}
This follows from the fact that $\pi$ is a homomorphism; comparing the top-left entry in the equation $\pi(\sigma\tau) = \pi(\sigma)\pi(\tau)$ yields the equality $a(\sigma \tau) = a(\sigma)a(\tau) + b(\sigma)c(\tau)$. By Lemma \ref{admodm}, we know that $a(G_F) \subset 1 + \mathfrak{m}_f\mathbb{T}$, and as such, $b(\sigma)c(\tau)$ must be inside of $\mathfrak{m}_f$ for all $\sigma, \tau \in G_F$. 
\end{proof}
\begin{definition}
Let $B$ denote the $\mathbb{T}$-submodule of $\mathbb{K}$ generated by all elements of the form $b(\sigma)$ for $\sigma \in G_F$, and $C$ the analogous submodule using the elements $c(\sigma)$.
\end{definition}
\begin{lemma}\label{jsinj}
There are well-defined injective maps
\begin{align*}
j_B \colon \emph{Hom}_{\mathbb{T}}( B/ \mathfrak{m}_f B, \mathbb{T} / \mathfrak{m}_f ) \to H^1( G_F, \mathbb{Q}_p(\chi) ) &: f \mapsto \chi \cdot (f \circ b); \\
j_C \colon \emph{Hom}_{\mathbb{T}}( C/ \mathfrak{m}_f C, \mathbb{T} / \mathfrak{m}_f ) \to H^1( G_F, \mathbb{Q}_p(\chi) ) &: g \mapsto g \circ c.
\end{align*}
\end{lemma}
\begin{proof}
The off-diagonal entries in the equation $\pi(\sigma\tau) = \pi(\sigma)\pi(\tau)$ yield respectively
\[
b(\sigma\tau) = d(\tau)b(\sigma) + a(\sigma)b(\tau) \quad \text{and} \quad c(\sigma\tau) = a(\tau)c(\sigma) + d(\sigma)c(\tau).
\]
Reducing mod $\mathfrak{m}_f$ and using Lemma \ref{admodm}, these equations reduce to the cocycle conditions for $\chi \cdot b$ and $c$ respectively, readily showing well-definedness. To show injectivity, we observe that $G_L$ is contained in the kernel of every coboundary. So if $G_L \subset \text{ker}( f \circ b )$, because also $b(\gamma) = 0$ for $\gamma \in G_F \setminus G_L$, it readily follows that $f \circ b$ would be completely trivial; the same argument works for $g \circ c$.
\end{proof}
We continue to exploit the knowledge that $\pi$ is nearly ordinary to obtain certain local information about the entries $b$ and $c$, which we will later use to deduce further global properties of the modules $B$ and $C$. What ensues is a subtle dance between global properties and information about local restrictions, which turns out to provide us with all the necessary conclusions.
\begin{lemma}\label{eigvalcongs}
For $i \in \{ 1,2 \}$, the maps $\epsilon_i, \delta_i \colon G_{\mathfrak{p}_i} \to \mathbb{T}^{\times} / \mathfrak{m}_f \cong \mathbb{Q}_p^{\times}$ are equal to $1, \chi$, or $\chi, 1$.
\end{lemma}
\begin{proof}
In the proof of Lemma \ref{admodm} we showed that $\text{Tr}(\pi(\sigma)) \equiv 1 + \chi(\sigma) \mod \mathfrak{m}_f$ and because $\text{det}(\pi(\text{Frob}_{\mathfrak{l}})) = \chi(\mathfrak{l})$ for all primes $\mathfrak{l} \nmid p$, completely similarly we may extend this formula to all of $G_F$. We have thus shown that $\epsilon_i(\sigma) + \delta_i(\sigma) = 1 + \chi(\sigma) \mod \mathfrak{m}_f$ and $\epsilon_i(\sigma)\delta_i(\sigma) = \chi(\sigma) \mod \mathfrak{m}_f$. For $\sigma \in G_L \cap G_{\mathfrak{p}_i}$, this is readily rewritten as $(\epsilon_i(\sigma) - 1)^2 \equiv (\delta_i(\sigma) - 1)^2 \equiv 0 \mod \mathfrak{m}_f$. This shows that $\epsilon_i(\sigma) \equiv \delta_i(\sigma) \equiv 1 \mod \mathfrak{m}_f$ on $G_L \cap G_{\mathfrak{p}_i}$. This determines these characters on an index 2 subgroup, which leaves only two possibilities; $1$ or $\chi$. Since we can choose $\sigma = \text{Frob}_{\mathfrak{p}_i} \in G_{\mathfrak{p}_i} \setminus G_L$ to find that $\epsilon_i(\sigma) + \delta_i(\sigma) \equiv 0 \mod \mathfrak{m}_f$, it is clear that each choice must occur exactly once. 
\end{proof}
\begin{prop}\label{bccobounds}
For each $i \in \{ 1, 2 \}$, at least one of the following must hold:
\begin{itemize}
\item $\chi \cdot b \mod \mathfrak{m}_f B$ is a coboundary when restricted to $G_{\mathfrak{p}_i}$;
\item $c \mod \mathfrak{m}_f C$ is a coboundary when restricted to $G_{\mathfrak{p}_i}$.
\end{itemize}
\end{prop}
\begin{proof}
The change of basis matrix that changes $\pi$ into the upper triangular form from Theorem \ref{heckealgrep} must satisfy for all $\sigma \in G_{\mathfrak{p}_i}$ the equality
\[
\begin{pmatrix} a(\sigma) & b(\sigma) \\ c(\sigma) & d(\sigma) \end{pmatrix} \begin{pmatrix} x & y \\ z & w \end{pmatrix} = \begin{pmatrix} x & y \\ z & w \end{pmatrix} \begin{pmatrix} \epsilon_i(\sigma) & * \\ 0 & \delta_i(\sigma) \end{pmatrix}.
\]
Comparing the top left entries, we obtain that $b(\sigma) = \frac{x}{z}\big(\epsilon_i(\sigma) - a(\sigma)\big)$. Similarly, comparing the bottom left entries, we obtain that $c(\sigma) = \frac{z}{x}\big(\epsilon_i(\sigma) - d(\sigma)\big)$. Using Lemma \ref{eigvalcongs} above in combination with Lemma \ref{admodm}, it follows that either $\epsilon_i(\sigma) - a(\sigma) \in \mathfrak{m}_f$ or $\epsilon_i(\sigma) - d(\sigma) \in \mathfrak{m}_f$ for all $\sigma \in G_{\mathfrak{p}_i}$. Suppose the former. Then let $\tau = \text{Frob}_{\mathfrak{p}_i} \in G_{\mathfrak{p}_i}$, so that $\epsilon_i(\tau) - d(\tau) \in \mathbb{T}^{\times}$. This shows that $\frac{z}{x} = c(\tau) \cdot \big(\epsilon_i(\sigma) - d(\sigma)\big)^{-1} \in C$, and as such, $c(\sigma) = \frac{z}{x}\big(\epsilon_i(\sigma) - d(\sigma)\big) \equiv \frac{z}{x}\big( 1 - \chi(\sigma) \big) \mod \mathfrak{m}_f C$, which shows that $c$ is a coboundary $\mod \mathfrak{m}_f C$. The other case is completely analogous.
\end{proof}
\begin{gevolg}\label{bcbothcobound}
Let $\{ i, j \} = \{ 1, 2 \}$. If $\chi \cdot b \mod \mathfrak{m}_f B$ is a coboundary when restricted to $G_{\mathfrak{p}_i}$, then $c \mod \mathfrak{m}_f C$ is a coboundary when restricted to $G_{\mathfrak{p}_j}$.
\end{gevolg}
\begin{proof}
By Proposition \ref{bccobounds} above, it suffices to show that it cannot occur that $b$ or $c$ is a coboundary $\mod \mathfrak{m}_f$ when restricted to both $G_{\mathfrak{p}_1}$ and $G_{\mathfrak{p}_2}$. Let us therefore suppose the contrary for $c$; we claim that it is then a coboundary globally. Indeed, similarly to the proof of Lemma \ref{gfpH1dim}, the map
\[
H^1(G_F, C / \mathfrak{m}_f C) \to H^1(G_{\mathfrak{p}_1}, C / \mathfrak{m}_f C) \oplus H^1(G_{\mathfrak{p}_2}, C / \mathfrak{m}_f C)
\]
is an isomorphism. It follows that there exists some $\lambda \in C / \mathfrak{m}_f C$ such that $c(\sigma) \equiv \lambda \cdot ( 1 - \chi(\sigma) ) \mod \mathfrak{m}_f C$. In particular, $0 = c(\gamma) \equiv 2\lambda \mod \mathfrak{m}_f C$ which shows that $C / \mathfrak{m}_f C = 0$. By Nakayama's Lemma, it follows from this that $C = 0$ globally, and as such, $c$ must be the zero-cocycle. However, $\pi$ is irreducible; this is a contradiction and completes the proof.
\end{proof}
\begin{prop}\label{bcfreerk1}
The modules $B$ and $C$ are free $\mathbb{T}$-modules of rank 1.
\end{prop}
\begin{proof}
Using the same argument as Lemme 4 in \cite{BelChen}, it follows that both $B$ and $C$ are $\mathbb{T}$-modules of finite type. The images of the maps $j_B$ and $j_C$ are 1-dimensional inside $H^1( G_F, \mathbb{Q}_p(\chi) )$, because Corollary \ref{bcbothcobound} above implies that both $b$ and $c$ will be locally trivial at precisely one of the two places above $p$, which cuts out a 1-dimensional subspace by Lemma \ref{gfpH1dim}. By the injectivity established by Lemma \ref{jsinj}, using that $\mathbb{T} / \mathfrak{m}_f \mathbb{T} \cong \mathbb{Q}_p$, it follows that $\text{Hom}_{\mathbb{T}}( B/ \mathfrak{m}_f B, \mathbb{Q}_p )$ is 1-dimensional. In other words, $B/ \mathfrak{m}_f B$ is generated by a single element so by Nakayama's Lemma, the same must then hold for $B$ itself; the argument for $C$ is analogous.
\end{proof}
\begin{gevolg}\label{Tmodbasis}
There exists a basis of $\mathbb{K}^2$ such that the image of $\pi$ takes values in $\mathbb{T}^2$ and is upper triangular $\mod \mathfrak{m}_f$. The $\mathbb{T}$-module spanned by these basis vectors is $G_F$-stable.
\end{gevolg}
\begin{proof}
By Proposition \ref{bcfreerk1} above, we can find an element $b_0 \in B$ generating the $\mathbb{T}$-module $B$. Now consider the basis $\{ b_0e_1, e_2 \}$, in which $\pi$ looks like
\[
\pi( \sigma ) = \begin{pmatrix} a(\sigma) & b(\sigma)b_0^{-1} \\ c(\sigma)b_0 & d(\sigma) \end{pmatrix}.
\]
By Lemma \ref{bcinm}, it follows that $c(\sigma)b_0 \in \mathfrak{m}_f$ for all $\sigma \in G_F$. This means that $\pi$ takes values in $\mathbb{T}^2$, and as such, it stabilises the $\mathbb{T}$-lattice $M = \langle b_0e_1, e_2 \rangle$.
\end{proof}
We now rescale our original choice of basis vectors as in the corollary above, to omit $b_0$ from any future calculations. In addition, in view of Corollary \ref{bcbothcobound}, we may assume that $b \mod \mathfrak{m}_f$ is a coboundary when restricted to $G_{\mathfrak{p}_2}$, whereas $c \mod \mathfrak{m}_f$ is a coboundary when restricted to $G_{\mathfrak{p}_1}$.
\begin{prop}\label{finalbasis}
Up to a rescaling a basis vector by some $\lambda \in \mathbb{Q}_p^{\times}$, the map $\pi$ is a lift of $\rho_{\eta}$.
\end{prop}
\begin{proof}
From Lemma \ref{admodm} and the proof of Corollary \ref{Tmodbasis}, we already know that $a \equiv 1 \mod \mathfrak{m}_f$, and $d \equiv \chi \mod \mathfrak{m}_f$ and $c \equiv 0 \mod \mathfrak{m}_f$. It thus suffices to show that $b(\sigma) \equiv \lambda \cdot \eta \mod \mathfrak{m}_f$ for some $\lambda \in \mathbb{Q}_p^{\times}$. Recall that $\eta$ is a generator for the 1-dimensional subspace of $H^1(G_F, \mathbb{Q}_p(\chi))$ of cocycles that vanish on the decomposition group $G_{\mathfrak{p}_2} \subset G_F$. Since we assume that $b \mod \mathfrak{m}_f$ is a coboundary when restricted to $G_{\mathfrak{p}_2}$ and further $b(\gamma) = 0$, it readily follows that $b$ vanishes completely on $G_{\mathfrak{p}_2}$. As a result, $b \mod \mathfrak{m}_f = \lambda \cdot \eta$ and since $b \mod \mathfrak{m}_f$ cannot be trivial when restricted to $G_{\mathfrak{p}_1}$ as a result of Corollary \ref{bcbothcobound}, it follows that even $\lambda \in \mathbb{Q}_p^{\times}$. 
\end{proof}
To conclude that $\pi$ is now actually a nearly ordinary deformation of $\rho_{\eta}$, it remains to identify lines inside $\mathbb{T}^2$ on which suitable restrictions of $\pi$ act scalar.
\begin{stelling}\label{heckerep}
Consider $\pi$ from Theorem \ref{heckealgrep} in any basis with the property that the conditions from Proposition \ref{finalbasis} are satisfied. Then $\pi$ defines a nearly ordinary deformation of $\rho_{\eta}$.
\end{stelling}
\begin{proof}
It suffices to exhibit free direct summands $L_i$ for $i \in \{ 1, 2 \}$ of rank 1 inside $\mathbb{T}^2$ which are stable under the restriction $\pi|_{G_{\mathfrak{p}_i}}$ and which lift $e_i$. By Theorem \ref{heckealgrep}, for $i \in \{ 1, 2 \}$, we can find matrices with
\[
\begin{pmatrix} a(\sigma) & b(\sigma) \\ c(\sigma) & d(\sigma) \end{pmatrix} \begin{pmatrix} x_i & y_i \\ z_i & w_i \end{pmatrix} = \begin{pmatrix} x_i & y_i \\ z_i & w_i \end{pmatrix} \begin{pmatrix} \epsilon_i(\sigma) & * \\ 0 & \delta_i(\sigma) \end{pmatrix},
\]
yielding the equalities $b(\sigma) = \frac{x_i}{z_i}\big(\epsilon_i(\sigma) - a(\sigma)\big)$ and $c(\sigma) = \frac{z_i}{x_i}\big(\epsilon_i(\sigma) - d(\sigma)\big)$. By our choices for when $b$ and $c$ are respectively trivial $\mod \mathfrak{m}_f$, it follows from Lemma \ref{admodm} that $\epsilon_2 \equiv 1 \mod \mathfrak{m}_f$ and that also $\epsilon_1 \equiv \chi \mod \mathfrak{m}_f$. Using the same argument as in the proof of Proposition \ref{bccobounds}, we show that $z_1/x_1 \in \mathfrak{m}_f$ and $x_2/z_2 \in \mathfrak{m}_f$. We thus set
\[
L_1 = \left\langle \begin{pmatrix} 1 \\ z_1 / x_1 \end{pmatrix} \right\rangle \quad \text{and} \quad L_2 = \left\langle \begin{pmatrix} x_2/z_2 \\ 1 \end{pmatrix} \right\rangle.
\]
These lines are free of rank 1 inside $\mathbb{T}^2$ and fixed by $\pi|_{G_{\mathfrak{p}_i}}$ by construction. Further, by the above, they reduce to $e_1$ and $e_2$ respectively, completing the proof.
\end{proof}

\subsection{The modularity theorem}

The goal of this section will be to prove an isomorphism $R_{\rho_{\eta}}^{\text{no}} \cong \mathbb{T}$. We have already constructed the map, as Theorem \ref{heckerep} claims that there exists a nearly ordinary deformation $\pi$ of $\rho_{\eta}$ to $\mathbb{T}$. By the universal property of $R_{\rho_{\eta}}^{\text{no}}$, this induces a map $\mathcal{T} \colon R_{\rho_{\eta}}^{\text{no}} \to \mathbb{T}$ that induces this deformation from $\rho^{\text{univ}}$.
\begin{lemma}\label{modularitymapsurj}
The map $\mathcal{T} \colon R_{\rho_{\eta}}^{\emph{no}} \to \mathbb{T}$ is surjective.
\end{lemma}
\begin{proof}
Let $\Lambda = \mathbb{Q}_p \llbracket X, Y, Z \rrbracket$ be as defined in Section 2.2 and 3.1 in \cite{BDS}. Both $R_{\rho_{\eta}}^{\text{no}}$ and $\mathbb{T}$ carry a natural $\Lambda$-algebra structure and the map $\mathcal{T}$ defined above is generally $\Lambda$-linear. Since $\mathbb{T}$ is generated over $\Lambda$ by the operators $T_{\mathfrak{l}}$, $\langle \mathfrak{l} \rangle$ and $U_x$ for $x \in \mathcal{O}_F \otimes \mathbb{Z}_p$, it suffices to show that these are contained in the image of $\mathcal{T}$. We use the defining relation that $\pi = \mathcal{T} \circ \rho^{\text{univ}}$ to show for $\mathfrak{l} \nmid p$ that
\[
\mathcal{T}\big( \text{Tr}( \rho^{\text{univ}}(\text{Frob}_{\mathfrak{l}} ) ) \big) = \text{Tr}\big( \mathcal{T}( \rho^{\text{univ}}(\text{Frob}_{\mathfrak{l}} ) ) \big) = \text{Tr}( \pi( \text{Frob}_{\mathfrak{l}} ) ) = T_{\mathfrak{l}}. 
\]
Similarly,
\[
\mathcal{T}\big( \text{det}( \rho^{\text{univ}}(\text{Frob}_{\mathfrak{l}} ) ) \big) = \text{det}\big( \mathcal{T}( \rho^{\text{univ}}(\text{Frob}_{\mathfrak{l}} ) ) \big) = \text{det}( \pi( \text{Frob}_{\mathfrak{l}} ) ) = \langle \mathfrak{l} \rangle \text{Nm}(\mathfrak{l}). 
\]
It now suffices to consider the operators $U_x$ for $x \in F \otimes \mathbb{Z}_p \cong F_{\mathfrak{p}_1} \times F_{\mathfrak{p}_2}$. Use the stable lines for $\rho^{\text{univ}}$ and $\pi$ to construct bases. If we then let $\Delta$ denote the bottom right entry of $\rho^{\text{univ}}$, we then obtain for $x \in F_{\mathfrak{p}_1}^{\times}$ that $\mathcal{T}\big( \Delta(x) \big) = \delta_1(x) = U_x$. The top-left entry yields the same result for $x \in F_{\mathfrak{p}_2}^{\times}$, completing the proof.
\end{proof}
\begin{lemma}\label{ringthy1}
Let $k$ be a field and further let $(A, m_A)$ and $(B,m_B)$ be local $k$-algebras. Suppose that $\emph{dim}_k( m_A / m_A^2 ) = \emph{dim}(B) < \infty$ and that there is a surjective map of $k$-algebras $A \to B$. Then $A$ and $B$ are both regular local rings with the same finite Krull dimension.
\end{lemma}
\begin{proof}
The existence of a surjective map $A \to B$ implies that $\text{dim}(A) \geq \dim(B)$ and similarly for the tangent spaces. Krull's principal ideal theorem implies that the dimension of the tangent space is bounded below by the Krull dimension of the ring. Combining these two observations with the given equality of dimensions quickly yields that everything must be equal, completing the proof.
\end{proof}
\begin{prop}\label{ringthy3}
Let $k$ be a field and let $(A, m_A)$ and $(B, m_B)$ be Noetherian regular local $k$-algebras with the same finite Krull dimension. Then every surjective map $A \to B$ must be an isomorphism. 
\end{prop}
\begin{proof}
This is just commutative algebra. A proof can be found in the author's PhD thesis.
\end{proof}
\begin{stelling}\label{R=T}
The map $\mathcal{T} \colon R_{\rho_{\eta}}^{\emph{no}} \to \mathbb{T}$ is an isomorphism.
\end{stelling}
\begin{proof}
Corollary \ref{notandim} showed that $\text{dim}(t_{\rho_{\eta}}^{\text{no}}) = 3$ and its is well-known that $\mathbb{T}$ is equidimensional of dimension 3. By Lemma \ref{modularitymapsurj}, the map $\mathcal{T}$ is surjective. Now Lemma \ref{ringthy1} implies that both $R_{\rho_{\eta}}^{\text{no}}$ and $\mathbb{T}$ are regular of the same Krull dimension. As they are also Noetherian, Proposition \ref{ringthy3} implies that the surjective map $\mathcal{T}$ must in fact be an isomorphism, completing the proof. 
\end{proof}

\newpage

\section{Analytic proof}\label{proof2}

Let $\log_p$ denote the Iwasawa branch of the $p$-adic logarithm. Our first goal of this section is to rewrite the expressions
\[
\frac{2}{w_1w_2} \log_p \Theta( D_1, D_2 ) \quad \text{and} \quad \frac{2}{w_1w_2} \log_p \Theta_p( D_1, D_2 )
\]
into a form more closely related to the terms appearing in the Fourier expansion of the Hilbert Eisenstein series $E_{1,\chi}^{(p)}$. Throughout, we will write $\text{Gal}( F / \mathbb{Q} ) = \langle \sigma \rangle$, but we will also use the notation $x' \colonequals \sigma(x)$ for $x \in F$. After that, we will extract the $p$-adic modular form that is associated with a particular nearly ordinary deformation as considered in the previous section and compute its coefficients explicitly. Finally, we compute its diagonal restriction, take its derivative and compute its ordinary projection. The vanishing of this expression will yield a proof of Theorem \ref{giamconj2}. 

\subsection{From quaternions to ideals}

This subsection uses an approach similar to the one described in Section 2 from \cite{GZrefined}. We will be brief, and for many of the details, we refer to their work. 

Fix two embeddings $\alpha_1 \colon \mathcal{O}_1 \to R_q$ and $\alpha_2 \colon \mathcal{O}_2 \to R_q$. This turns $B_q$ into an $L$-vector space as follows. Let $x \in K_1$ and $y \in K_2$. Then the action of the element $xy \in L$ on some $\gamma \in B_q$ is defined by $xy * \gamma = \alpha_1(x)\gamma\alpha_2(y)$ and we extend this definition to all of $L$ by $\mathbb{Q}$-linearity. Since both $L$ and $B_q$ are 4-dimensional $\mathbb{Q}$-vector spaces, $B_q$ becomes a $1$-dimensional $L$-vector space and a $2$-dimensional $F$-vector space. 

Proposition 2.3 in \cite{GZrefined} shows the existence of an $F$-linear quadratic form $\text{det}_F : B_q \to F^+$ that is uniquely characterised by requiring that
\[
\text{Tr}_{F / \mathbb{Q}}(\text{det}_F(\gamma)) = \text{Nm}( \gamma )
\]  
for all $\gamma \in B_q$. We define the \emph{reflex ideal} associated with the embeddings $\alpha_1, \alpha_2$ as the intersection with $F$ of the kernel of the composition
\[
L \to B_{q} \to B_{q} / \Pi \cong \mathbb{F}_{q^2}, 
\]
where the first \emph{map} is only additive, and where $\Pi \in B_{q \infty}$ is an element of norm $q$ as provided in Section 2.2 in \cite{anphil}. By the commutativity of the rightmost ring, this composition is actually a ring morphism and as such, it defines an ideal in $L$. We assume that this reflex ideal is given by $\mathfrak{q}_1$.

The $p$-adic theta function is defined as an infinite product over the units in some maximal order inside a quaternion algebra. In order to relate this to Hilbert Eisenstein series, we describe a construction that turns counting quaternions with various properties into counting ideals of $L$. Choose some isomorphism of $L$-vector spaces $\iota \colon B_q \to L$. Naturally, $\iota$ is highly non-canonical, but we will still use it to define an $L$-ideal associated to some $b \in R_q$ as $$I_b \colonequals \iota(b) / \iota(R_q).$$ The ideal $I_b$ is both integral and independent of the choice of isomorphism $\iota \colon B_q \to L$. Combining Lemma 2.5, Lemma 2.16 and Lemma 2.22 in \cite{GZrefined}, we obtain the following.
\begin{prop}\label{idealnorm}
The ideal $I_b$ satisfies
\[
\emph{Nm}_{L / F}(I_b) = \emph{det}_F(b) \mathfrak{q}_1^{-1} \mathcal{D}_F.
\]
\end{prop}

Any attempt at constructing a bijection between quaternions and ideals using only one choice of embeddings is obstructed by the simple fact that $\iota(R)$, and as such $I_b$, will always be in the same ideal class. We must therefore take into account the action of the Picard groups on the embeddings, as defined in Corollary 30.4.23 in \cite{voight}. We will write $\iota[c_1,c_2]$ for an isomorphism of 1-dimensional $L$-vector spaces $B_q \to L$ where $B_q$ is equipped with the $L$-vector space structure induced by the embeddings $[c_1] \cdot \alpha_1$ and $[c_2] \cdot \alpha_2$, where $[c_1] \in \text{Pic}(K_1)$ and $[c_2] \in \text{Pic}(K_2)$. Further, let $I[c_1,c_2]_b$ denote the ideal associated to $b$ using the embedding $\iota[c_1,c_2]$ and let $\text{det}_F[c_1,c_2]$ be the resulting $F$-bilinear quadratic form.

The following bijection will be key in rewriting the $\Theta$-series into a more useful form.
\begin{stelling}\label{quatcount}
For any totally positive $\nu \in F^+$, the association $(b,[c_1],[c_2]) \mapsto I[c_1,c_2]_b$ establishes a bijection between the set of $(b,[c_1],[c_2]) \in (\mathcal{O}_1^{\times} \setminus R_q \ / \ \mathcal{O}_2^{\times}) \times \emph{Pic}(K_1) \times \emph{Pic}(K_2)$ with the property that $\emph{det}_F[c_1,c_2](b) = \nu$ and the set of integral ideals $I \subset \mathcal{O}_L$ such that $\emph{Nm}_{L/F}(I) = (\nu) \mathfrak{q}_1^{-1} \mathcal{D}_F$.
\end{stelling}
\begin{proof}
This is a rephrased version of Corollary 2.24 in \cite{GZrefined}. Fundamentally, one uses the exact sequence below; this is the same sequence that was used critically in Section 6 in the original paper \cite{GZ}.
\begin{figure}[h]
\centering
\begin{tikzcd}
  1 \rar & \{ \pm 1 \} \rar & \mathcal{O}_1^{\times} \times \mathcal{O}_2^{\times} \rar & \mathcal{O}_L^{\times} \rar & \mathcal{O}_F^{\times,+} \ar[out=-90, in=90, looseness=0.33]{dllll} \\
  \text{Pic}(K_1) \times \text{Pic}(K_2) \rar & \text{Pic}(L) \rar & \text{Pic}(F)^+ \rar &  \{ \pm 1 \} \rar & 1.
\end{tikzcd}
\end{figure}
\end{proof}

\subsection{Rewriting \texorpdfstring{$\Theta( D_1, D_2 )$}{Theta(D1,D2)}}

Recall that
\[
\Theta( D_1, D_2 ) \colonequals \prod_{\substack{ \text{Pic}(K_1) \cdot \tau_1 \\ \text{Pic}(K_2) \cdot \tau_2 } } \prod_{b \in R_{q}[1/p]_1^{\times}} [\tau_1, \tau_1', b\tau_2, b\tau_2'].
\]
It turns out that the cross-ratio is connected to the form $\text{det}_F$ if we introduce the form $\text{det}_F'$, which is the form induced by the embeddings $\overline{\alpha_1}$ and $\alpha_2$. It also admits a more explicit description.

\begin{lemma}
For any $b \in B$, the numbers $\emph{det}_F(b)$ and $\emph{det}_F'(b)$ are $\emph{Gal}(F / \mathbb{Q})$-conjugates.
\end{lemma}
\begin{proof}
It suffices to show that the composite $\sigma \circ \text{det}_F \colon B \to F$ satisfies the defining property of $\text{det}_F'$. This is an easy check and is left to the reader.
\end{proof}

\begin{prop}\label{detFformula}
Let $\tau_i, \tau_i'$ be the two fixed points in $\mathcal{H}_p$ for the image of $\alpha_i( \mathcal{O}_i )$ under any choice of splitting $B_q \otimes \mathbb{Z}_p \to M_2( \mathbb{Q}_p )$. Then
\[
[\tau_1, \tau_1', b\tau_2, b\tau_2'] = -\frac{\emph{det}_F(b)}{\emph{det}_F'(b)}.
\]
\end{prop}
\begin{proof}
One may check that the explicit formula
\[
f(b) = \text{Nm}(b) \frac{ ( \tau_1 - b \tau_2 )( \tau_1' - b \tau_2' ) }{ ( \tau_1 - \tau_1' ) ( b \tau_2 - b \tau_2' ) }
\]
precisely satisfies the defining properties for $\text{det}_F$. The result now follows from the observation that for $\overline{\alpha_1}$, the points $\tau_1$ and $\tau_1'$ are swapped.
\end{proof}
\begin{stelling}\label{thetarw}
It holds that
\[
\frac{2}{w_1w_2} \log_p \Theta( D_1, D_2 ) = \lim_{n \to \infty} \sum_{ \substack{ \nu \in (\mathcal{D}_F^{-1}\mathfrak{q}_1)^+ \\ \emph{Tr}(\nu) = p^{2n} } } \log_p \left( \frac{\nu}{\nu'} \right) \cdot \rho( \nu \mathfrak{q}_1^{-1} \mathcal{D}_F ).
\]
\end{stelling}
\begin{proof}
By Proposition \ref{detFformula}, ignoring the sign by pairing each quaternion with its negative, we obtain
\[
\Theta( D_1, D_2 ) = \prod_{\substack{ [c_1], [c_2] } } \prod_{b \in R_{q}[1/p]_1^{\times}} \frac{\text{det}_F[c_1,c_2](b)}{\text{det}_F'[c_1,c_2](b)},
\]
For any $b \in R_{q}[1/p]_1^{\times}$, there exists some minimal $k \geq 0$ such that $p^k b \equalscolon B \in R_q$. This association induces a bijection
\[
R_{q}[1/p]_1^{\times} \xrightarrow{\sim} \bigsqcup_{k = 0}^{\infty} \left\{ B \in R_q \mid p \nmid B, \text{ Nm}(B) = p^{2k} \right\}.
\]
We define for any $n \geq 0$ the set 
\[
R_q(n) \colonequals \left\{ B \in R_q \mid \text{ Nm}(B) = p^{2n} \right\}.
\]
Now we observe that the association $b \mapsto p^{n-k}b$ induces a bijection
\[
\bigsqcup_{k = 0}^{n} \left\{ B \in R_q \mid p \nmid B, \text{ Nm}(B) = p^{2k} \right\} \xrightarrow{\sim} R_q(n).
\]
As such, taking the $p$-adic logarithm,
\[
\log_p \Theta( D_1, D_2 ) = \lim_{n \to \infty} \sum_{\substack{ [c_1], [c_2] } } \sum_{b \in R_q(n)} \log_p \left( \frac{\text{det}_F(b)}{\text{det}_F'(b)} \right).
\]
We switch the order of summation; instead of summing over all $b \in R_q(k)$ and recording its associated $\text{det}_F$-value, we will sum over each possible $\text{det}_F$-value and record how often it is reached by some $b \in R_q(k)$. Recalling that $b \in R_q(k)$ means that $\text{Tr}(\text{det}[c_1,c_2]_F(b)) = \text{Nm}(b) = p^{2k}$, we find
\[
\sum_{ \substack{ \nu \gg 0 \\ \text{Tr}(\nu) = p^{2n} } } \log_p \left( \frac{\nu}{\nu'} \right) \cdot \# \big\{ (b,[c_1],[c_2]) \in R_q(n) \times \text{Pic}(K_1) \times \text{Pic}(K_2) \mid \text{det}_F[c_1,c_2](b) = \nu \big\}.
\]
We solved this counting problem in Theorem \ref{quatcount}; taking care with units in $\mathcal{O}_1^{\times}$ and $\mathcal{O}_2^{\times}$, we write the above as
\[
\log_p \Theta( D_1, D_2 ) = \frac{w_1w_2}{2} \lim_{n \to \infty} \sum_{ \substack{ \nu \gg 0, \nu \in \mathcal{D}_F^{-1}\mathfrak{q}_1 \\ \text{Tr}(\nu) = p^{2n} } } \log_p \left( \frac{\nu}{\nu'} \right) \cdot \rho( \nu \mathfrak{q}_1^{-1} \mathcal{D}_F );
\]
this completes the proof.
\end{proof}
Recall that $\pi \in R_q$ denoted a quaternion with $\text{Nm}(\pi) = p$. Multiplication by $\pi$ induces a bijection
\[
R_q[1/p]_1^{\times} \colonequals \big\{ b \in R_q[1/p]^{\times} \mid \text{Nm}(b) = 1 \big\} \xrightarrow{\sim} \big\{ b \in R_q[1/p]^{\times} \mid \text{Nm}(b) = p \big\}\equalscolon R_q[1/p]_p^{\times},
\] 
with inverse map given by multiplication by $\overline{\pi}/p$. We now have the following result.
\begin{gevolg}\label{thetaprw}
It holds that
\[
\frac{2}{w_1w_2} \log_p \emph{Nm}\left( \Theta_p(D_1, D_2) \right) = \lim_{n \to \infty} \sum_{ \substack{ \nu \in (\mathcal{D}_F^{-1}\mathfrak{q}_1)^+ \\ \emph{Tr}(\nu) = p^{2n+1} } } \log_p \left( \frac{\nu}{\nu'} \right) \cdot \rho( \nu \mathfrak{q}_1^{-1} \mathcal{D}_F ),
\]
\end{gevolg}
\begin{proof}
We simply observe that that
\[
\frac{\Theta(\tau_1, \tau_1'; \pi \tau_2)}{\Theta(\tau_1, \tau_1'; \pi \tau_2')} = \prod_{b \in R_{q}[1/p]_1^{\times}} \frac{ \big(\tau_1 - b \pi \tau_2\big)\big(\tau_1' - b \pi \tau_2'\big) }{ \big(\tau_1 - b\pi \tau_2'\big)\big(\tau_1' - b \pi \tau_2\big)} = \prod_{b \in R_q[1/p]_p^{\times}} \frac{ (\tau_1 - b\tau_2)(\tau_1' - b\tau_2') }{ (\tau_1 - b\tau_2')(\tau_1' - b\tau_2)};
\]
now we obtain our result through identical reasoning as in the proof of Theorem \ref{thetarw}.
\end{proof}

As our method to analyse the explicit values of the $p$-adic number $\Theta( D_1, D_2) / \Theta_p( D_1, D_2)$ ultimately only gives us an equality after taking the $p$-adic logarithm $\log_p$, to obtain a genuine equality, we must analyse the number of factors of $p$ occurring on both sides of the equation in Theorem \ref{giamconj2} separately. This is established in the following proofs.

\begin{lemma}\label{pnnuprop}
Let $\nu \in \mathcal{D}_F^{-1}$ be such that $\emph{Tr}(\nu) = p^n$ for some positive integer $n$ and suppose further that $v_p( \emph{Nm}(\nu) )$ is odd or that $v_{\mathfrak{p}_1}(\nu) \neq v_{\mathfrak{p}_2}(\nu)$. Then $p^n \mid \nu$. 
\end{lemma}
\begin{proof}
Let us explicitly write
\[
\nu \sqrt{D} = \frac{ x + p^n \sqrt{D} }{2}, \quad \text{so that} \quad \text{Nm}( \nu \sqrt{D} ) = \frac{ x^2 - p^{2n}D }{4}.
\]
Write $x = p^ky$ where $p \nmid y$. If $k \geq n$, we are done. If not, we find that
\[
\nu \sqrt{D} = p^k \frac{ y + p^{n-k} \sqrt{D} }{2} \quad \text{and} \quad \text{Nm}( \nu \sqrt{D} ) = p^{2k} \frac{ y^2 - p^{2n - 2k}D }{4}.
\]
By assumption, the fractions still contain factors of $p$. One deduces that $p \mid y$; this is a contradiction.
\end{proof}

\begin{prop}\label{giamatp}
It holds that
\[
\frac{\pm 2}{w_1w_2} v_p\left( \frac{\Theta( D_1, D_2 )}{\Theta_p( D_1, D_2 )} \right) = \sum_{\substack{x^2 < D \\ x^2 \equiv D \emph{ mod } 4N}} \delta(x) v_p \left( F\left( \frac{D-x^2}{4N} \right) \right).
\]
\end{prop}
\begin{proof}
The proof of Theorem \ref{thetarw} shows that, if we neglect to apply $\log_p$, we obtain
\[
\Theta( D_1, D_2 )^{ \frac{ \pm 2 }{w_1w_2} } = \lim_{n \to \infty} \prod_{ \substack{ \nu \in (\mathcal{D}_F^{-1}\mathfrak{q}_1)^+ \\ \text{Tr}(\nu) = p^{2n} } } \left( \frac{\nu}{\nu'} \right)^{\rho( \nu \mathfrak{q}_1^{-1} \mathcal{D}_F )}
\]
and similarly for $\Theta_p( D_1, D_2 )$. We claim that the $\mathfrak{p}_1$-adic valuation of each term in the limit is constant. Indeed, only terms with $v_{\mathfrak{p}_1}(\nu) \neq v_{\mathfrak{p}_1}(\nu') = v_{\mathfrak{p}_2}(\nu)$ can contribute to the $\mathfrak{p}_1$-adic valuation. By Lemma \ref{pnnuprop}, this means that only those $\nu$ lifted from trace 1 can contribute. As indeed $\rho( p^{2n} \nu \mathfrak{q}_1^{-1} \mathcal{D}_F ) = \rho( \nu \mathfrak{q}_1^{-1} \mathcal{D}_F )$, we conclude that
\begin{align*}
\frac{\pm 2}{w_1w_2} v_{\mathfrak{p}_1}\left( \Theta(D_1, D_2) \right) &= \sum_{ \substack{ \nu \in (\mathcal{D}_F^{-1}\mathfrak{q}_1)^+ \\ \text{Tr}(\nu) = 1 } } \rho( \nu \mathfrak{q}_1^{-1} \mathcal{D}_F ) \big( v_{\mathfrak{p}_1}(\nu) - v_{\mathfrak{p}_1}(\nu') ) \\ 
 &= \sum_{ \substack{ \nu \in (\mathcal{D}_F^{-1}\mathfrak{p}_1\mathfrak{q}_1)^+ \\ \text{Tr}(\nu) = 1 } } \rho( \nu \mathfrak{q}_1^{-1} \mathcal{D}_F ) v_{\mathfrak{p}_1}(\nu) - \sum_{ \substack{ \nu \in (\mathcal{D}_F^{-1}\mathfrak{p}_1\mathfrak{q}_2)^+ \\ \text{Tr}(\nu) = 1 } } \rho( \nu \mathfrak{q}_2^{-1} \mathcal{D}_F ) v_{\mathfrak{p}_1}(\nu).
\end{align*}
On the other hand, because either $v_{\mathfrak{p}_1}( \nu ) = 0$ or $v_{\mathfrak{p}_2}( \nu ) = 0$ for $\nu$ of trace 1, it always holds that $\rho( p^{2n+1} \nu \mathfrak{q}_1^{-1} \mathcal{D}_F ) = 0$, which has the consequence that $v_{\mathfrak{p}_1}( \Theta_p( D_1, D_2 ) ) = 0$. For similar reasons, only those $\nu$ for which $v_{\mathfrak{p}_1}(\nu)$ is even can contribute to the sum expressing $v_p( \Theta(D_1, D_2) )$. This means that the prime $p$ divides the quantity $\text{Nm}(\nu \sqrt{D})/N = (D - x^2)/4N$ an odd number of times; whence its $F$-value will be a power of the prime $p$ by definition. The agreement between the exponent in the definition of the $F$-value and the function $\rho$ has been shown before in the proof of Proposition \ref{thm1}. Now repeat for $\mathfrak{p}_2$ and add.
\end{proof}

\subsection{Extracting \texorpdfstring{$a_{\nu}$}{av} from \texorpdfstring{$\widetilde{\rho}$}{rho}}

In this subsection we will extract the $p$-adic modular form that is associated with a particular nearly ordinary deformation as considered in the previous section, of which we have proved its modularity in Theorem \ref{R=T}. Explicitly, we choose the equivalence class of the lift
\[
\widetilde{\rho} = \left( 1 + \epsilon \begin{pmatrix} a & 0 \\ 0 & d \end{pmatrix} \right) \begin{pmatrix} 1 & \chi \eta \\ 0 & \chi \end{pmatrix}
\]
where $a = -d \in \text{Hom}( G_F, \mathbb{Q}_p )$. Note that this is in fact a deformation, because our choice $c = 0$ forces $a,d \in \text{Hom}( G_F, \mathbb{Q}_p )$ as a result of Lemma \ref{xwmap}. The lines $\langle e_1 \rangle$ and $\langle e_2 \rangle$ are fixed by $G_{\mathfrak{p}_1}$ and $G_{\mathfrak{p}_2}$ respectively. Using that $\chi$ is trivial on the inertia subgroups, it follows that the quotient characters are given by
\[
\mu_{\mathfrak{p}_1} = d|_{I_{\mathfrak{p}_1}} = 1 - \text{log}_p(\chi_p)\epsilon \quad \text{and} \quad \mu_{\mathfrak{p}_2} = a|_{I_{\mathfrak{p}_2}} = 1+\text{log}_p(\chi_p) \epsilon.
\]
After identifying $( \mathcal{O}_F \otimes \mathbb{Z}_p )^{\times} \cong \mathcal{O}_{F_{\mathfrak{p}_1}}^{\times} \times \mathcal{O}_{F_{\mathfrak{p}_2}}^{\times} \cong I_{\mathfrak{p}_1} \times I_{\mathfrak{p}_2}$, the weight character is given by $\mu_{\mathfrak{p}_1} \times \mu_{\mathfrak{p}_2}$. In particular, the weight character for the diagonal restriction of this modular form can be computed as the composition
\[
(\mathbb{Z} \otimes \mathbb{Z}_p)^{\times} \xrightarrow{\Delta} ( \mathcal{O}_F \otimes \mathbb{Z}_p )^{\times} \cong \mathcal{O}_{F_{\mathfrak{p}_1}}^{\times} \times \mathcal{O}_{F_{\mathfrak{p}_2}}^{\times} \to \mathbb{Q}_p[\epsilon],
\]
where $\Delta$ denotes the diagonal embedding. We explicitly compute that
\[
x \mapsto ( x, x ) \mapsto \mu_{\mathfrak{p}_1}(x) \mu_{\mathfrak{p}_2}(x) = (-1 + \log_p(x)\epsilon)(1 + \log_p(x)\epsilon) = -1.
\]
In particular, the diagonal restriction is of constant weight. This shows that the above deformation describes an infinitesimal family of modular forms in the \emph{anti-parallel} weight direction.

Recall from Proposition \ref{gfhomdim} that
\[
\text{Hom}(G_F, \mathbb{Q}_p) \cong \text{ker}\Big( \text{Hom}(F_{\mathfrak{p}_1}^{\times} \times F_{\mathfrak{p}_2}^{\times}, \mathbb{Q}_p) \to \text{Hom}(\mathcal{O}_F[1/p]^{\times}, \mathbb{Q}_p) \Big)
\]
is a 1-dimensional $\mathbb{Q}_p$-vector space. This kernel is spanned by the map
\[
\left( F \otimes \mathbb{Z}_p \right)^{\times} \xrightarrow{\sim} F_{\mathfrak{p}_1}^{\times} \times F_{\mathfrak{p}_2}^{\times} \to \mathbb{Q}_p
\]
sending $(x,y)$ to $\text{log}_p( x y )$. Indeed, any element $x$ from $F$ embeds as $(x,\sigma(x))$, and if $u \in \mathcal{O}_F^{\times}[1/p]$ then $u \cdot \sigma(u) = \text{Nm}^F_{\mathbb{Q}}(u) \in \pm p^{\mathbb{Z}}$. As such, its image under the Iwasawa brach of the $p$-adic logarithm vanishes, as by definition $\log_p(p) = 0$. We now explicitly choose $a = -d$ to equal this map, which can also be written as $\log_p \circ \chi_p^{\text{cyc}}$, where $\chi_p^{\text{cyc}}$ denotes the $p$-adic cyclotomic character.

We now extract from the traces of $\widetilde{\rho}$ evaluated at $\text{Frob}_{\mathfrak{l}}$ for $\mathfrak{l}$ a prime of $F$ a morphism $\varphi \colon \mathbb{T} \to \mathbb{Q}[\epsilon]$. Recall that $\mathbb{T}$ is generated by the operators $T_{\mathfrak{l}}$ for all primes $\mathfrak{l}$ of $F$ coprime to $p$, and the operators $U_{\pi_1}$ and $U_{\pi_2}$ where $\pi_1, \pi_2 \in \mathbb{A}_F^{\times}$ are local uniformisers at $\mathfrak{p}_1$ and $\mathfrak{p}_2$ respectively. 

\begin{prop}
Let $\mathfrak{l} \nmid p$ be a prime ideal of $F$. Then
\[
\varphi( T_{\mathfrak{l}} ) = 
\begin{cases}
2 &\text{if $\chi(\mathfrak{l}) = 1$;} \\
2\log_p(\emph{Nm}(\mathfrak{l}))\epsilon &\text{if $\chi(\mathfrak{l}) = -1$.}
\end{cases}
\]
\end{prop}
\begin{proof}
By Theorem \ref{heckealgrep}, we have $\varphi( T_{\mathfrak{l}} ) = \text{Tr}(\tilde{\rho}(\text{Frob}_{\mathfrak{l}}))$ as long as $\mathfrak{l} \nmid p$. It is easy to see that
\[
\text{Tr}( \widetilde{\rho}( \tau ) ) = 1 + \chi( \tau ) + (1 - \chi( \tau ) )\log_p( \chi_p^{\text{cyc}}( \tau ) )
\]
for all $\tau \in G_F$. Now we must split cases. If the prime ideal $\mathfrak{l} \nmid p$ splits in the field extension $L / F$, then $\text{Frob}_{\mathfrak{l}}$ is trivial in $\text{Gal}(L/F)$ and as such, $\chi( \text{Frob}_{\mathfrak{l}} ) = 1$ and the expression for the trace above yields the result immediately. If the prime ideal $\mathfrak{l} \nmid p$ is inert in the field extension $L / F$, then $\text{Frob}_{\mathfrak{l}}$ is nontrivial in $\text{Gal}(L/F)$ and as such, $\chi( \text{Frob}_{\mathfrak{l}} ) = -1$. We then find that
\[
\text{Tr}( \tilde{\rho}(\text{Frob}_{\mathfrak{l}}) ) = 2\log_p ( \chi_p^{\text{cyc}}( \text{Frob}_{\mathfrak{l}} ) )\epsilon = 2\log_p(\text{Nm}(\mathfrak{l}) )\epsilon;
\]
this completes the proof.
\end{proof}
Let $\pi_1, \pi_2 \in \mathbb{A}_F^{\times}$ be local uniformisers at $\mathfrak{p}_1$ and $\mathfrak{p}_2$ respectively, being trivial at all other places. Finding the images of $U_{\pi_1}$ and $U_{\pi_2}$ under the morphism $\varphi$ works slightly differently.

\begin{prop}\label{Upiimg}
Let $\pi_1, \pi_2 \in \mathbb{A}_F$ be as above. Then
\[
\varphi(U_{\pi_1}) = -1 + \log_p( \pi_1 )\epsilon \quad \text{and} \quad \varphi(U_{\pi_2}) = 1 + \log_p( \pi_2 )\epsilon.
\]
\end{prop}
\begin{proof}
By Theorem \ref{heckealgrep}, we obtain the images of $U_{\pi}$ and $U_{\pi'}$ not as traces of $\tilde{\rho}$, but as the image of the local characters $\mu_{\mathfrak{p}_1}$ and $\mu_{\mathfrak{p}_2}$. The local character in the first case is
\[
\mu_{\mathfrak{p}_1}(\pi, 1) = \chi( \pi ) + \chi( \pi ) d(\pi ) \epsilon = -1 + \log_p( \pi ) \epsilon.
\]
Completely similarly, $\mu_{\mathfrak{p}_2}(1, \pi') = 1 + a(\pi' ) \epsilon = 1 + \log_p( \pi' ) \epsilon$, completing the proof.
\end{proof}
In order to continue with higher powers of prime ideals, we must also determine the images of the diamond operators. Fortunately, this is straightforward.
 
\begin{lemma}
For any prime ideal $\mathfrak{l} \nmid p$ of $F$, it holds that
\[
\varphi( \langle \mathfrak{l} \rangle \emph{Nm}(\mathfrak{l}) ) = \chi( \emph{Frob}_{\mathfrak{l}} ).
\]
\end{lemma}
\begin{proof}
Again by Theorem \ref{heckealgrep}, the image of $\text{Frob}_{\mathfrak{l}}$ has determinant $\varphi( \langle \mathfrak{l} \rangle \text{Nm}(\mathfrak{l}) )$. In our case, since we kept the determinant constant as $a + d = 0$, this is simply $\chi( \text{Frob}_{\mathfrak{l}} )$.
\end{proof}
\begin{prop}\label{Tellnimgs}
Let $\mathfrak{l} \nmid p$ be a prime ideal of $F$ and $n \geq 0$ an integer. Then
\[
\varphi( T_{\mathfrak{l}^n} ) = 
\begin{cases}
n+1 &\text{if $\chi(\mathfrak{l}) = 1$;} \\
(n+1)\log_p( \emph{Nm}(\mathfrak{l}) )\epsilon &\text{if $\chi(\mathfrak{l}) = -1$ and $n$ is odd;} \\
1 &\text{if $\chi(\mathfrak{l}) = -1$ and $n$ is even.}
\end{cases}
\]
Further, it holds that
\begin{align*}
\varphi( U_{\pi_1^n} ) &= (-1)^n ( 1 - n\log_p( \pi_1 ) \epsilon ); \\
\varphi( U_{\pi_2^n} ) &= 1 + n\log_p( \pi_2 )\epsilon. 
\end{align*}
\end{prop}
\begin{proof}
We remind the reader of the essential recursion relation
\[
T_{\mathfrak{l}^{n+1}} = T_{\mathfrak{l}^n}T_{\mathfrak{l}} - \langle \mathfrak{l} \rangle \text{Nm}(\mathfrak{l}) T_{\mathfrak{l}^{n-1}}
\]
for $\mathfrak{l} \nmid p$, whereas simply $U_{\pi^n} = U_{\pi}^n$ for the places above $p$.
\begin{itemize}
\item If the prime ideal $\mathfrak{l} \nmid p$ splits in the field extension $L / F$, then $\text{Tr}( \tilde{\rho}(\text{Frob}_{\mathfrak{l}}) ) = 2$ and $\chi( \text{Frob}_{\mathfrak{l}} ) = 1$. We obtain the recursion
\[
T(n+1) = 2T(n) - T(n-1) \quad \text{with} \quad L(0) = 1, \quad L(1) = 2.
\]
This is easily solved and yields $L(n) = n + 1$ for all $n \geq 0$.
\item If the prime ideal $\mathfrak{l} \nmid p$ is inert in the field extension $L / F$, then $\text{Tr}( \tilde{\rho}(\text{Frob}_{\mathfrak{l}}) ) = 2\log_p( \text{Nm}(\mathfrak{l}) )\epsilon$ and $\chi( \text{Frob}_{\mathfrak{l}} ) = -1$. We obtain the recursion
\[
L(n+1) = 2\log_p( \text{Nm}(\mathfrak{l}) )\epsilon \cdot L(n) + L(n-1)
\]
with $L(0) = 1$ and $L(1) = 2\log_p( \text{Nm}(\mathfrak{l}) )\epsilon$. Since $\epsilon^2 = 0$, this results in
\[
L(2n) = 1 \quad \text{and} \quad L(2n-1) = 2n\log_p( \text{Nm}(\mathfrak{l}) )\epsilon \quad \text{for all $n \geq 1$}.
\]
\end{itemize}
For the operators $U_{\pi_1}$ and $U_{\pi_2}$, we may simply raise the result from Proposition \ref{Upiimg} to the appropriate power to obtain the claimed formula.
\end{proof}
 
\begin{gevolg}
Let $\mathfrak{l} \nmid p$ be a prime ideal of $F$ and let $n \geq 0$ be an integer. Then
\[
\varphi( T_{\mathfrak{l}^n} ) = \rho( \mathfrak{l}^n ) + \frac{1}{2}(n+1)\big(1 - \chi(\mathfrak{l}^n)\big) \log_p( \emph{Nm}(\mathfrak{l}) ) \epsilon,
\]
where $\rho( I )$ denotes the number of integral ideals of $L$ with norm equal to $I \subset \mathcal{O}_F$. 
\end{gevolg}
\begin{proof}
Indeed, we have seen before, and it is easy convince oneself, that
\[
\rho( \mathfrak{l}^n ) = 
\begin{cases}
n+1 &\text{if $\chi(\mathfrak{l}) = 1$;} \\
0 &\text{if $\chi(\mathfrak{l}) = -1$ and $n$ is odd;} \\
1 &\text{if $\chi(\mathfrak{l}) = -1$ and $n$ is even.}
\end{cases}
\]
These quantities match the integral parts of $\varphi( T_{\mathfrak{l}^n} )$ that we found above. As for the infinitesimal part, we get no contribution precisely when $\chi( \mathfrak{l}^n ) = 1$, and as such, the expression $1 - \chi( \mathfrak{l}^n )$ is twice the indicator function for the case $\chi( \mathfrak{l} ) = -1$ and $n$ is odd. Combining these two parts yields the corollary.
\end{proof}
\begin{gevolg}\label{phiJcor}
Let $J \subset \mathcal{O}_F$ be any ideal coprime to $p$. Then
\[
\varphi( T_J ) = \rho(J) + \frac{1}{2}\sum_{\mathfrak{l}^n \| J} \Big( (n+1)\big(1 - \chi(\mathfrak{l}^n)\big) \rho(J / \mathfrak{l}^n) \Big) \log_p(\emph{Nm}(\mathfrak{l}))\epsilon.
\]
\end{gevolg}
\begin{proof}
Using the definition
\[
T_J \colonequals \prod_{ \mathfrak{l}^n \| J } T_{\mathfrak{l}^n},
\]
one may write out, keeping in mind that $\epsilon^2 = 0$, that
\begin{align*}
\varphi( T_J ) &= \prod_{ \mathfrak{l}^n \| J } \Big( \rho( \mathfrak{l}^n ) + \frac{1}{2}(n+1)\big(1 - \chi(\mathfrak{l}^n)\big) \log_p( \text{Nm}(\mathfrak{l}) ) \epsilon \Big) \\ 
 &= \prod_{ \mathfrak{l}^n \| J } \rho( \mathfrak{l}^n ) + \frac{1}{2}\sum_{\mathfrak{l}^n \| J} (n+1)\big(1 - \chi(\mathfrak{l}^n)\big)\log_p(\text{Nm}(\mathfrak{l}))\epsilon \prod_{ \mathfrak{r}^m \| J / \mathfrak{l}^n } \rho( \mathfrak{r}^m );
\end{align*}
this yields the corollary after recalling the multiplicativity of $\rho$.
\end{proof}
Let us now define for any integral ideal $J$ coprime to $p$ a positive integer $\mathcal{F}(J)$ by
\[
\log_p(\mathcal{F}(J)) \colonequals \frac{1}{2}\sum_{\mathfrak{l}^n \| J} \Big( (n+1)\big(1 - \chi(\mathfrak{l}^n)\big) \rho(J / \mathfrak{l}^n) \Big) \log_p(\text{Nm}(\mathfrak{l})).
\]
For the sake of brevity and clarity, we will henceforth refer to those prime powers $\mathfrak{l}^n \| J$ with $\chi( \mathfrak{l}^n ) = -1$, as the \emph{special primes} of an ideal $J \subset \mathcal{O}_F$. Note here that $\mathfrak{p}_1$ and $\mathfrak{p}_2$ can also be special primes, if we relax the condition that $J$ be coprime to $p$, as we will soon be forced to do.
 
\begin{prop}\label{calFprops}
Let $J \subset \mathcal{O}_F$ be any integral ideal coprime to $p$. Then
\[
\varphi(T_J) = \rho(J) + \log_p (\mathcal{F}(J)) \epsilon.
\]
In addition, $\mathcal{F}(J)$ is a power of a single rational prime. If $J$ is a primitive ideal, then it even holds that $\mathcal{F}(J) = F( \emph{Nm}(J) )^2$, where $F$ is as defined in the introduction.
\end{prop}
\begin{proof}
The first claim follows directly from Corollary \ref{phiJcor} and the definition of $\mathcal{F}(J)$. For the second, we must observe that the only summands in the expression defining $\mathcal{F}(J)$ that could possibly contribute are those for the special primes of $J$. If there are no such primes, then $\mathcal{F}(J) = 1$. If there is more than special prime, one of which being $\mathfrak{l}^n$, then its contribution will also vanish because all $\rho( J / \mathfrak{l}^n ) = 0$, as the existence of a special prime obstructs an ideal from being a norm from $L$. We conclude that $\mathcal{F}(J) = 1$ in that case too. Only the case in which there is a unique special prime remains, proving that $\mathcal{F}(J)$ is a power of the underlying rational prime $\ell$, as claimed. Finally, our primitivity assumption forces all primes dividing $J$ to split in $F / \mathbb{Q}$ and to all lie above different rational primes, and as such, the prime factorisation of $J$ in $F$ matches the prime factorisation of its norm in $\mathbb{Q}$. In the proof of Proposition \ref{thm1}, we saw that the exact exponent of $\ell$ occuring in the expression $F( \text{Nm}(J) )$ can also be written as $(n+1)\rho(J / \mathfrak{l}^n) / 2$, completing the proof.
\end{proof}
For any integral ideal $J$ of $F$, we let $\widetilde{J}$ denote its $p$-deprivation, which is obtained by removing all factors of $\mathfrak{p}_1$ and $\mathfrak{p}_2$ from the factorisation of $J$. Recall that $\mathfrak{q}_1$ denotes the reflex ideal associated with the embeddings $\alpha_i \colon \mathcal{O}_i \to R_q$. Let $E_{1,\chi}^{(p)}(\epsilon)$ be the Hilbert modular form associated with the morphism $\varphi \colon \mathbb{T} \to \mathbb{Q}[\epsilon]$ computed above, which is an $h_F^+$-tuple of $q$-expansions and consider its component corresponding to the narrow ideal class of $\mathcal{D}_F^{-1}\mathfrak{q}_1$. Finally, let $a_{\nu}$ for any $\nu \in  (\mathcal{D}_F^{-1}\mathfrak{q}_1)^+$ denote one of its coefficients. 
\begin{stelling}
For any $\nu \in (\mathcal{D}_F^{-1}\mathfrak{q}_1)^+$, let $J_{\nu}$ denote the ideal $\nu \mathcal{D}_F \mathfrak{q}_1^{-1}$. Then
\[
a_{\nu} = (-1)^{v_{\mathfrak{p}_1}(\nu)} \big( \rho(\widetilde{J_{\nu}}) + \log_p (\mathcal{F}(\widetilde{J_{\nu}})) \epsilon - \rho(\widetilde{J_{\nu}})\log_p( \nu / \nu')\epsilon \big).
\]
\end{stelling}
\begin{proof}
To compute $a_{\nu}$ for some $\nu \in (\mathcal{D}_F^{-1}\mathfrak{q}_1)^+$, we must consider the id\`ele $\alpha = \nu d \pi_{\mathfrak{q}_1}^{-1}$, where $\pi_{\mathfrak{q}_1}$ is any id\`ele that equals 1 everywhere away from $\mathfrak{q}_1$, where it is a uniformiser, and where $d \in \mathbb{A}_F^{\times}$ is such that it generates the ideal $\mathcal{D}_F$. Let $\widetilde{\nu}$ denote the id\`ele that is equal to $\nu$ everywhere away from $p$, where it is equal to 1. Then $\nu = \tilde{\nu} \nu_{\mathfrak{p}_1} \nu_{\mathfrak{p}_2}$. We may then compute that
\begin{align*}
\varphi( T_{\alpha} ) &= \varphi( T_{ \widetilde{\nu}d\pi_{\mathfrak{q}_1}^{-1} } ) \varphi( U_{\nu_{\mathfrak{p}_1}} ) \varphi( U_{\nu_{\mathfrak{p}_2}} ) \\
 &= \varphi( T_{ \widetilde{J_{\nu}} } ) \cdot (-1)^{v_{\mathfrak{p}_1}(\nu_{\mathfrak{p}_1})}( 1 - \log_p( \nu_{\mathfrak{p}_1} )\epsilon ) \cdot ( 1 + \log_p( \nu_{\mathfrak{p}_2} )\epsilon ) \\
 &= (-1)^{v_{\mathfrak{p}_1}(\nu_{\mathfrak{p}_1})} \big(\rho(\widetilde{J_{\nu}}) + \log_p (\mathcal{F}(\widetilde{J}_{\nu}) \epsilon \big) \big( 1 - \log_p( \nu_{\mathfrak{p}_1} / \nu_{\mathfrak{p}_2})\epsilon \big) \\
 &= (-1)^{v_{\mathfrak{p}_1}(\nu_{\mathfrak{p}_1})} \big( \rho(\widetilde{J_{\nu}}) + \log_p (\mathcal{F}(\widetilde{J}_{\nu}) \epsilon - \rho(\widetilde{J_{\nu}})\log_p( \nu / \nu')\epsilon \big);
\end{align*}
this is precisely the theorem. We used here that $v_{\mathfrak{p}_2}$ can be identified with $v'_{\mathfrak{p}_1}$, since under the isomorphism $\mathcal{O}_F \otimes \mathbb{Z}_p \cong \mathcal{O}_{F_{\mathfrak{p}_1}} \times \mathcal{O}_{F_{\mathfrak{p}_2}}$, the element $\nu$ is sent to $(\nu, \sigma(\nu) ) = (\nu, \nu')$. 
\end{proof}

\subsection{Proof of Theorem \ref{giamconj2}}
We take the diagonal restriction of the form obtained in the previous section;
\[
\Delta E_{1,\chi}^{(p)}(\epsilon) = \sum_{n = 1}^{\infty} \Big( \sum_{\substack{\nu \in (\mathcal{D}_F^{-1}\mathfrak{q}_1)^+ \\ \text{tr}(\nu) = n}} a_{\nu} \Big) q^n.
\]
Taking its derivative with respect to the weight amounts to considering only the $\epsilon$-part, which yields
\[
\frac{d}{d\epsilon} \Delta E_{1,\chi}^{(p)}(\epsilon) = \sum_{n = 1}^{\infty} \Big( \sum_{\substack{\nu \in (\mathcal{D}_F^{-1}\mathfrak{q}_1)^+ \\ \text{tr}(\nu) = n}} (-1)^{v_{\mathfrak{p}_1}(\nu)} \big(\log_p( \mathcal{F}( \widetilde{J_{\nu}})) - \rho( \widetilde{J_{\nu}} )\log_p( \nu / \nu' ) \big) \Big)  q^n.
\]
\begin{prop}\label{s2Nprop}
The object $\frac{d}{d\epsilon} \Delta E_{1,\chi}^{(p)}(\epsilon)$ is an overconvergent $p$-adic modular form of weight 2. Its ordinary projection $e^{\emph{ord}}\left( \frac{d}{d\epsilon} \Delta E_{1,\chi}^{(p)}(\epsilon) \right)$ is a classical modular form in $S_2(\Gamma_0(N))$. 
\end{prop}
\begin{proof}
We have seen before that the weight character for $\Delta E_{1,\chi}^{(p)}(\epsilon)$ is constant, and because for the $\epsilon = 0$-specialisation its weight is simply $1 + 1 = 2$, the result will be of constant weight 2. By subtracting a constant family, Lemma 2.1 in \cite{DPV1} yields that its derivative is also an overconvergent $p$-adic modular form of weight 2. By Coleman's Classicality Theorem, which can be found as Theorem 6.1 in \cite{coleman}, its ordinary projection is of slope $0 < 1$ and hence classical. Further, it is a cusp form because $E_{1,\chi}^{(p)}$ is a $p$-adic cusp form. For the level, since we took the ideal $\mathcal{D}_F\mathfrak{q}_1^{-1}$, the tame level of our diagonal restriction will be exactly $q$. The level of its ordinary projection is then obtained by multiplying its tame level by $p$. Combining all of this, we obtain an object in $S_2(\Gamma_0(N))$, as claimed.
\end{proof}
Explicitly, if we apply the operator $e^{\text{ord}}$, we obtain, using that $n!$ is even for $n \geq 2$ and the fact that we have $p$-adic convergence for all even terms by Theorem \ref{thetarw},
\begin{align*}
a_1\left( e^{\text{ord}} \left( \frac{d}{d\epsilon} \Delta E_{1,\chi}^{(p)}(\epsilon) \right) \right) &= \lim_{n \to \infty} a_{p^{n!}} \left( \frac{d}{d\epsilon} \Delta E_{1,\chi}^{(p)}(\epsilon) \right) \\ 
 &= \lim_{n \to \infty} \sum_{\substack{\nu \in (\mathcal{D}_F^{-1}\mathfrak{q}_1)^+ \\ \text{tr}(\nu) = p^{2n}}} (-1)^{v_{\mathfrak{p}_1}(\nu)} \big(\log_p( \mathcal{F}( \widetilde{J_{\nu}})) - \rho( \widetilde{J_{\nu}} )\log_p( \nu / \nu' ) \big).
\end{align*}
Now define
\begin{align*}
A &\colonequals \lim_{n \to \infty} \sum_{\substack{\nu \in (\mathcal{D}_F^{-1}\mathfrak{q}_1)^+ \\ \text{tr}(\nu) = p^{2n}}} (-1)^{v_{\mathfrak{p}_1}(\nu)} \rho( \widetilde{J_{\nu}} )\log_p( \nu / \nu' ); \\
B &\colonequals \lim_{n \to \infty} \sum_{\substack{\nu \in (\mathcal{D}_F^{-1}\mathfrak{q}_1)^+ \\ \text{tr}(\nu) = p^{2n}}} (-1)^{v_{\mathfrak{p}_1}(\nu)} \log_p( \mathcal{F}( \widetilde{J_{\nu}}) ).
\end{align*}
For the sake of brevity, we extend the definition of $v_p$ to $F$ by setting it equal to $v_{\mathfrak{p}_1} \times v_{\mathfrak{p}_2}$.
\begin{prop}
It holds that
\[
A = \frac{2}{w_1w_2} \log_p \Theta(D_1, D_2) - \frac{2}{w_1w_2} \log_p \Theta_p(D_1, D_2).
\]
\end{prop}
\begin{proof}
We note that since $\chi( J_{\nu} ) = \chi( \mathcal{D}_F ) \chi( \mathfrak{q}_1 ) = (-1)^2 = 1$, the partities of $v_{\mathfrak{p}_1}(J_{\nu})$ and $v_{\mathfrak{p}_2}(J_{\nu})$ being different means that $\chi( \widetilde{J_{\nu}} ) = -1$, and as such, $\rho( \widetilde{J_{\nu}} ) = 0$. Hence we may write
\[
A = \lim_{n \to \infty} \sum_{\substack{\nu \in (\mathcal{D}_F^{-1}\mathfrak{q}_1)^+ \\ \text{tr}(\nu) = p^{2n} \\ v_p(\nu) \equiv (0,0)}}\rho( \widetilde{J_{\nu}} )\log_p( \nu / \nu' ) - \lim_{n \to \infty} \sum_{\substack{\nu \in (\mathcal{D}_F^{-1}\mathfrak{q}_1)^+ \\ \text{tr}(\nu) = p^{2n} \\ v_p(\nu) \equiv (1,1)}} \rho( \widetilde{J_{\nu}} )\log_p( \nu / \nu' ),
\]
the congruences being mod 2. For the first term, one may observe that $\rho( \widetilde{J_{\nu}} ) = \rho( J_{\nu} )$. In fact, $\rho( J_{\nu} ) = 0$ unless $v_p(\nu) \equiv (0,0) \mod 2$, and as a result, we may even write the first term as
\[
\lim_{n \to \infty} \sum_{\substack{\nu \in (\mathcal{D}_F^{-1}\mathfrak{q}_1)^+ \\ \text{tr}(\nu) = p^{2n}}}\rho( J_{\nu} )\log_p( \nu / \nu' ) = \frac{2}{w_1w_2} \log_p \Theta(D_1,D_2),
\]
where we appealed to Theorem \ref{thetarw}. For the second term, one may observe that $p \mid \nu$, and as such, we may make that substitution, further using that $\rho( \widetilde{ J_{\nu} } ) = \rho( J_{p\nu} )$ in this case, to obtain
\[
\lim_{n \to \infty} \sum_{\substack{\nu \in (\mathcal{D}_F^{-1}\mathfrak{q}_1)^+ \\ \text{tr}(\nu) = p^{2n} \\ v_p(\nu) \equiv (1,1)}} \rho( \widetilde{J_{\nu}} )\log_p( \nu / \nu' ) = \lim_{n \to \infty} \sum_{\substack{\nu \in (\mathcal{D}_F^{-1}\mathfrak{q}_1)^+ \\ \text{tr}(\nu) = p^{2n-1}}} \rho( J_{\nu} )\log_p( \nu / \nu' ) = \frac{2}{w_1w_2}\log_p \Theta_p(D_1, D_2),
\]
where we appealed to Corollary \ref{thetaprw} and where we were allowed to omit the bottom subscript for the same reason as before. 
\end{proof}
\newpage
\begin{prop}
It holds that
\[
B = \sum_{\emph{Nm}(\mathfrak{a}) = N} \sum_{\substack{\nu \in (\mathcal{D}_F^{-1}\mathfrak{a})^+ \\ \emph{tr}(\nu) = 1}} \delta(\mathfrak{a}) \log_p( F(\emph{Nm}(J_{\nu}) / p) ).
\]
\end{prop}
\begin{proof}
First note that, by Lemma \ref{calFprops}, it holds that $\mathcal{F}(\widetilde{J_{\nu}}) = 1$ as soon as $\chi( \widetilde{ J_{\nu} } ) = 1$, because this implies that the number of special primes is even, and thus in particular not one. Since $\chi( J_{\nu} ) = 1$, it follows that precisely one of $\mathfrak{p}_1$ and $\mathfrak{p}_2$ must be special to get a non-zero contribution to the sum. Since $\nu$ contains the full $p$-part of $J_{\nu}$, this implies that $v_p( \text{Nm}(\nu) )$ must be odd. By Lemma \ref{pnnuprop}, it must hold that $p^{2n} \mid \nu$, where $\text{tr}(\nu) = p^{2n}$. In other words, $\nu = p^{2n} \mu$, where $\text{tr}(\mu) = 1$. It follows that all contributing summands to the $n$-th term in the limit are lifted from those $\nu$ of unit trace. In fact, since $\widetilde{J_{p^{2n}\nu}} = \widetilde{J_{\nu}}$ and $v_p( J_{p^{2n}\nu} ) \equiv v_p( J_{\nu} ) \mod 2$, each summand induced by some $\nu$ of unit trace is independent of the prime exponent $n$. It follows that the limit is equal to its first term;
\[
B = \sum_{\substack{\nu \in (\mathcal{D}_F^{-1}\mathfrak{q}_1)^+ \\ \text{tr}(\nu) = 1 \\ v_p(\text{Nm}(\nu)) \text{ odd}}} (-1)^{v_{\mathfrak{p}_1}( J_{\nu} )} \log_p( \mathcal{F}( \widetilde{J_{\nu}}) ).
\]
Note that the ideal $\widetilde{J_{\nu}}$ is always primitive, because the element $\nu \sqrt{D}$ is of the form $(x + \sqrt{D})/2$ and as such, no rational prime can divide it. As it is prime to $p$ by definition, $\mathcal{F}( \widetilde{J_{\nu}} ) = F( \text{Nm}(\widetilde{J_{\nu}}) )^2$ in all cases by Proposition \ref{calFprops}. Note further that $v_p( \text{Nm}(J_{\nu}) )$ must be odd, so $v_p( \text{Nm}(J_{\nu}) / p )$ will be even. As such, we have $F( \text{Nm}(\widetilde{J_{\nu}}) ) = F( \text{Nm}(J_{\nu}) / p )$. Further note that if $v_p( \text{Nm}(J_{\nu}) )$ were even, dividing by $p$ would make $p$ a special prime of $\text{Nm}(J_{\nu}) / p$. As such, its $F$-value must be a power of $p$, of which the $p$-adic logarithm vanishes. Since contributing $\nu$ must contain a factor of $\mathfrak{p}_1$ or $\mathfrak{p}_2$, we have proved that
\[
B =  2\sum_{\substack{\nu \in (\mathcal{D}_F^{-1}\mathfrak{p}_1\mathfrak{q}_1)^+ \\
\text{ or } \nu \in (\mathcal{D}_F^{-1}\mathfrak{p}_2\mathfrak{q}_1)^+ \\ \text{tr}(\nu) = 1}} (-1)^{v_{\mathfrak{p}_1}( J_{\nu} )} \log_p( F(\text{Nm}(J_{\nu}) / p) ).
\]
Adding in those $\nu \in (\mathcal{D}_F^{-1}\mathfrak{q}_2)^+$ is the same as adding a term for every $\nu' \in (\mathcal{D}_F^{-1}\mathfrak{q}_1)^+$. For every non-zero term in the sum, we have that $v_{\mathfrak{p}_1}(J_{\nu'}) \not\equiv v_{\mathfrak{p}_1}(J_{\nu}) \mod 2$ and $\text{Nm}(J_{\nu'}) = \text{Nm}(J_{\nu})$. In other words, the summands for $\nu$ and $\nu'$ would agree up to a sign measured by both $\delta(\mathfrak{a})$ and $(-1)^{v_{\mathfrak{p}_1}(\nu)}$.
\end{proof}
\begin{proof}(of Theorem \ref{giamconj2})
By Proposition \ref{s2Nprop}, we have $e^{\text{ord}}\left( \frac{d}{d\epsilon} \Delta E_{1,\chi}^{(p)}(\epsilon) \right) \in S_2(\Gamma_0(N))$; we claim that it is even identically zero. Indeed, for $N \in \{ 6, 10 \}$, this space is zero. For $N = 22$, this space is 2-dimensional, containing two oldforms. From our explicit descriptions of the coefficients, it is not difficult to deduce that $a_{p^k} = (-1)^k a_1$; it is a quick check that for $p = 2, 11$, no such cuspforms exist in $S_2(\Gamma_0(22))$. We conclude that, in particular, $a_1( e^{\text{ord}}\left( \frac{d}{d\epsilon} \Delta E_{1,\chi}^p(\epsilon) \right) ) = 0$. This means that $A + B = 0$; written out, we find
\[
\frac{2}{w_1w_2} \left( \log_p \Theta(D_1, D_2) - \log_p \Theta_p(D_1, D_2) \right) = \sum_{\text{Nm}(a) = N} \sum_{\substack{\nu \in (\mathcal{D}_F^{-1}a)^+ \\ \text{tr}(\nu) = 1}} \delta(a) \log_p( F(\text{Nm}(J_{\nu}) / p) ).
\]
Finally, for $\nu = (x + \sqrt{D})/2\sqrt{D}$, it holds that
\[
\text{Nm}(J_{\nu}) / p = \frac{D-x^2}{4N}.
\]
This is precisely Theorem \ref{giamconj2} up to a sign and up to powers of $p$. Since Proposition \ref{giamatp} took care of the powers of $p$, this completes the proof.
\end{proof}

\newpage

\renewcommand{\baselinestretch}{0.85}\normalsize
\bibliographystyle{alpha}
\bibliography{cmvalsbiblio}
\renewcommand{\baselinestretch}{1}\normalsize

\end{document}